\documentclass[12pt,reqno]{amsart}
\calclayout
\usepackage{amsmath,amsthm,amssymb}
\usepackage{mathrsfs}
\usepackage{colonequals}
\usepackage{enumitem}
\usepackage{graphicx,subcaption}
\usepackage{todonotes}
\usepackage{mathtools}
\usepackage{booktabs}
\usepackage[printonlyused,withpage]{acronym}
\usepackage[only,llbracket,rrbracket]{stmaryrd}
\usepackage[justification=centering]{caption}

\usepackage[colorlinks=true,allcolors=blue,pdftitle={Numerical approach to the London equation of superconductivity}]{hyperref}

\usepackage[initials,nobysame,alphabetic,msc-links,backrefs]{amsrefs}
\usepackage{color}

\definecolor{darkgreen}{rgb}{0.13, 0.4, 0.13}

\DeclareMathOperator{\grad}{\nabla}
\DeclareMathOperator{\diver}{div}
\DeclareMathOperator{\curl}{curl}
\DeclareMathOperator{\R}{R}

\newcommand{\lep}{|\log\varepsilon|}
\newcommand{\ep}{\varepsilon}

\newcommand{\RR}{\mathbb R}

\newcommand{\loc}{\mathrm{loc}}

\newcommand{\ga}{\Gamma}

\newtheorem{theorem}{Theorem}

\newtheorem{corollary}{Corollary}[section]
\newtheorem{proposition}{Proposition}[section]
\newtheorem{lemma}[proposition]{Lemma}
\newtheorem{definition}{Definition}[section]

\newtheorem*{defi*}{Definition}

\renewcommand{\P}{\mathbb{P}}

\newcommand{\ex}{\mathrm{ex}}
\newcommand{\hex}{h_\ex}
\newcommand{\Hex}{H_\ex}
\newcommand{\Aex}{A_\ex}
\newcommand{\Hzero}{H_{0}}
\newcommand{\Azero}{A_{0}}
\newcommand{\Uzero}{U_{0}}
\newcommand{\Hzeroex}{H_{0,\ex}}
\newcommand{\Azeroex}{A_{0,\ex}}

\newcommand{\Hcr}{H_{c_1}}
\newcommand{\Bzero}{B_{0}}
\newcommand{\phizero}{\phi_{0}}
\newcommand{\BA}{B_A}
\newcommand{\BAzero}{B_{\Azero}}
\newcommand{\phiA}{\phi_A}
\newcommand{\phiAzero}{\phi_{\Azero}}
\newcommand{\gazero}{\ga_0}

\newcommand{\Hzerone}{\mathring{H}^1(\Omega)}
\newcommand{\HzeroCurl}{\mathring{H}(\curl,\Omega)}

\newcommand{\dsy}{\mathrm{d}s_y}
\DeclareMathOperator{\bcurl}{\textbf{curl}}
\newcommand{\gammatau}{\gamma_{\tau}}
\newcommand{\gammaN}{\gamma_{N}}
\newcommand{\gamman}{\gamma_{\nu}}

\newcommand{\hsolhom}{\dot{H}_{\diver=0}}
\newcommand{\meisconf}{\ensuremath{e^{i \hex \phizero}, \hex \Azero}}

\newcommand{\Hh}{H_{0,h}}
\newcommand{\Bh}{B_{0,h}}

\newcommand{\NED}{\mathrm{NED}}

\usepackage[normalem]{ulem}

\title{Numerical approach to the London Equation of superconductivity}

\author{Nicol\'{a}s Barnafi}
\address{Instituto de Ingenier\'ia Matem\'atica y Computacional \& Facultad de Ciencias Biológicas, Pontificia Universidad Cat\'olica de Chile, Vicu\~na Mackenna 4860, 7820436 Macul, Santiago, Chile; and Centro de Modelamiento Matemático, Santiago, Chile.}
\email{nicolas.barnafi@uc.cl}

\author{Ignacio Labarca-Figueroa}
\address{Instituto de Ingenier\'ia Matem\'atica y Computacional, Pontificia Universidad Cat\'olica de Chile, Vicu\~na Mackenna 4860, 7820436 Macul, Santiago, Chile}
\email{ignacio.labarca@uc.cl}

\author{Carlos Rom\'{a}n}
\address{Facultad de Matem\'aticas \& Instituto de Ingenier\'ia Matem\'atica y Computacional, Pontificia Universidad Cat\'olica de Chile, Vicu\~na Mackenna 4860, 7820436 Macul, Santiago, Chile}
\email{carlos.roman@uc.cl}

\date{\today}
\thanks{Funding information: ANID FONDECYT 1231593, Centro de Modelamiento Matemático, Proyecto Basal FB210005.}
\begin{document}
\begin{abstract}
In this work, we propose a general discretization strategy for solving the London equation for type-II superconductors in the whole space $\RR^3$. To compute the magnetic field $\Hzero$, we reformulate the problem for the magnetic potential as a transmission problem and discretize it through a nonstandard \ac{fem-bem} coupling. This formulation accounts for the unbounded exterior  domain as well as the interior domain without introducing an artificial truncation.

We then compute the vector field $\Bzero$, which arises from the Helmholtz--Hodge decomposition of the magnetic potential in the superconducting sample. This field enters the isoflux problem, which identifies the curves along which vortex nucleation first becomes energetically favorable in the \ac{gl} model of superconductivity. We recast the equations for $\Bzero$ using the mixed formulation of Kikuchi, in which the divergence-free constraint is imposed weakly, and discretize the resulting problem using a classical $H(\curl)$-conforming \ac{fe} discretization.

We validate our discretization strategy through convergence tests and conclude with an application to the isoflux problem. For a ball under a constant applied magnetic field, the unique maximizer is the diameter aligned with the field. For ellipsoids under a constant applied magnetic field aligned with their major axis, our computations provide numerical evidence of a different behavior in sufficiently elongated, cigar-shaped geometries: off-axis competitors reminiscent of U-shaped vortex configurations attain a larger isoflux ratio than the major axis. Since the major axis is therefore not a maximizer, any off-axis maximizer generates, by rotational symmetry, a continuous family of equivalent configurations, implying non-uniqueness and the presence of a degenerate rotational direction in the isoflux problem.

\end{abstract}
\maketitle 

\subsection*{Acronyms} 

\begin{acronym}
    \acro{amg}[AMG]{Algebraic Multigrid}
    \acro{bcs}[BCS]{Bardeen--Cooper--Schrieffer}
    \acro{bem}[BEM]{Boundary Element Method}
    \acro{bie}[BIE]{Boundary Integral Equation}
    \acro{dofs}[DoFs]{Degrees of Freedom}
    \acro{fe}[FE]{Finite Elements}
    \acro{fem}[FEM]{Finite Element Method}
    \acro{fem-bem}[FEM--BEM]{Finite Element--Boundary Element}
    \acro{gl}[GL]{Ginzburg--Landau}
    \acro{pde}[PDE]{Partial Differential Equation}
    \acrodefplural{pde}[PDEs]{Partial Differential Equation}
\end{acronym}

\section{Introduction}

\subsection{Problem and background}

The analysis of the first critical field $\Hcr$ in the three-dimensional magnetic \ac{gl} model of superconductivity leads to the study of the well-known \emph{London equation} for the magnetic field $\Hzero$,
\begin{equation}\label{Londoneq}
	\curl\curl (\Hzero - \Hzeroex) + \Hzero \mathbf{1}_\Omega = 0 \quad \mbox{in } \RR^3,
\end{equation}
supplemented by the conditions $\lim_{|x|\to \infty} |\Hzero - \Hzeroex | = 0$ and $[\Hzero - \Hzeroex ]_{\partial\Omega} = 0$.  Here, $\Omega$ represents the material sample, which we assume to be a bounded, simply connected subset of $\RR^3$ with $C^2$ boundary, and $\mathbf{1}_\Omega$ denotes its indicator function. The applied magnetic field is assumed to have the form $\Hex\coloneqq \hex\Hzeroex$, where $\Hzeroex$ is normalized so that $\|\Hzeroex\|_{L^\infty(\RR^3)}=1$ and represents its direction, while $\hex>0$ is a tunable parameter representing its intensity. Throughout the paper, we refer to $\Hzeroex$ as the normalized applied magnetic field. Moreover, $[\ \cdot\ ]_{\partial\Omega}$ denotes the jump across $\partial\Omega$. Both $\Hzero$ and $\Hzeroex$ are divergence-free vector fields, in view of their magnetic field structure.

\smallskip

Superconductivity was first discovered in 1911 by Heike Kamerlingh Onnes, who observed that mercury’s electrical resistance vanished abruptly at very low temperatures. This discovery emerged from his pioneering studies of matter at low temperatures, work that earned him the Nobel Prize in Physics in 1913. The experimental breakthrough revealed a new state of matter but initially lacked a theoretical explanation. The London equation, originally formulated in terms of the superconducting current by Fritz and Heinz London in \cite{London1935}, provided the first phenomenological framework for the Meissner effect, namely, the expulsion of magnetic fields from the interior of a superconductor below its critical temperature, with penetration only over a characteristic length scale. This phenomenon is one of the most striking features of superconductivity and underlies magnetic levitation.

\smallskip

The London equation for the magnetic field \eqref{Londoneq} naturally arises in the \ac{gl} framework. The \ac{gl} theory \cite{GinLan} deepened the phenomenological description of superconductivity by introducing an order parameter and a variational framework. Although not microscopic, it has proven to be an extraordinarily powerful theory, successfully accounting for most macroscopic features of superconductors. Moreover, it can be rigorously derived \cite{FHSS} as the limiting form of the microscopic \ac{bcs} theory~\cite{BCS}, in which superconductivity is explained by the formation of Cooper pairs of electrons.

\smallskip

In type-II superconductors, sufficiently strong applied magnetic fields can penetrate the sample through quantized vortex lines. In three dimensions, these vortices take the form of filaments, around which the superconducting current circulates and along whose cores superconductivity is locally lost. The first critical field $\Hcr$ marks the transition, at the level of global minimizers of the \ac{gl} functional, from configurations without vortex filaments for $\hex<\Hcr$ to configurations containing vortex filaments for $\hex>\Hcr$. Equation~\eqref{Londoneq} plays a central role in the computation of $\Hcr$ and in determining where vortex filaments first nucleate within the sample.

\smallskip
In \cite{Rom2}, the last author showed that
\begin{equation}\label{expansionHc1}
	\Hcr=\frac1{2\R_0}\lep+O(\log\lep)\quad \mbox{as}\ \ep \to 0,
\end{equation}
together with a strong characterization of minimizers both below and above $\Hcr$. Here:
\begin{itemize}
	\item $\ep>0$ is the inverse of the \emph{\ac{gl} parameter} $\kappa=\lambda/\xi$, where $\lambda$ is the London penetration depth, the characteristic length over which magnetic fields penetrate the superconductor, and $\xi$ is the coherence length, the characteristic length over which the superconducting order parameter varies. In this notation, type-II superconductors correspond to $0<\ep<\sqrt{2}$, while the asymptotic regime $\ep\to0$ describes extreme type-II superconductivity;
	
	\item the constant $\R_0$ corresponds to the value of the \emph{isoflux problem}, that is,
	$$
	\R_0= \sup_{\ga \in X} \R(\Gamma) \coloneqq\sup_{\ga \in X} \dfrac{1}{|\ga|}\int_\ga \Bzero\cdot d\ell,
	$$
	where $X$ is the set of curves in $\overline \Omega$ that do not self-intersect and that are either loops contained in $\Omega$ or have two distinct endpoints on $\partial \Omega$. The vector field $\Bzero$ is directly related to the magnetic field $\Hzero$ solving \eqref{Londoneq} and is obtained by solving
	\begin{equation}\label{eq:B0problem}
		\left\lbrace
		\begin{array}{rcll}
			\curl\curl \Bzero&=&\Hzero&\mathrm{in}\ \Omega\\
			\diver \Bzero&=&0&\mathrm{in}\ \Omega\\
			\Bzero\times \nu&=&0&\mathrm{on}\ \partial\Omega,
		\end{array}\right.
	\end{equation}
	where hereafter $\nu$ denotes the outer unit normal to $\partial\Omega$. Even though we will not make use of it, it is worth mentioning that it is not hard to prove that $\Bzero$ and $\Hzeroex-\Hzero$ differ only by the gradient of a harmonic function.
\end{itemize}

The last author, together with E. Sandier and S. Serfaty, \cites{RSS,RSS2} strengthened the expansion \eqref{expansionHc1} to
\begin{equation}\label{expansionHc1two}
	\Hcr=\frac1{2\R_0}\left(\lep+C_{\Omega,\gazero}\right)+o(1)\quad \mbox{as}\ \ep \to 0,
\end{equation}
when $\R_0$ is attained by a unique curve $\gazero$ that is non-degenerate in a suitable sense. Here, $C_{\Omega,\gazero}$ is a constant that depends on $\Omega$ and $\gazero$. In addition, above $\Hcr$, vortex filaments appear and exhibit a sequence of transitions, with vortex lines appearing one by one as the intensity of the applied magnetic field is increased: after passing $\Hcr$, there is one vortex; then, after increasing the applied field by an increment of order $\log |\log\ep|$, a second vortex line appears, and so on. These vortex lines accumulate near $\gazero$ and, after a suitable horizontal blow-up around $\gazero$, they minimize a next-order energy in the asymptotic limit $\ep\to 0$.

\smallskip

The only case thus far in which $\gazero$ has been identified is when $\Omega$ is a ball and $\Hzeroex$ is the constant vector field given by the unit vector in the $z$-direction. In this setting, an explicit solution of the London equation, and hence an explicit expression for $\Hzero$, was first obtained in \cite{Lon}. Using this explicit solution, Alama, Bronsard, and Montero subsequently identified $\Bzero$ in \cite{AlaBroMon} and showed that
$$
\Hcr\leq \frac{1}{2\R_0}\lep+o(\lep)\quad \mbox{as}\ \ep \to 0,
$$
in this special situation. This was later strengthened to
$$
\Hcr=\frac{1}{2\R_0}\lep+o(\lep)\quad \mbox{as}\ \ep \to 0,
$$
by Baldo, Jerrard, Orlandi, and Soner in \cite{BalJerOrlSon2}, for arbitrary domains and applied magnetic fields, together with a weak characterization of the behavior of minimizers both below and above $\Hcr$, later refined into the sharper expansions \eqref{expansionHc1} and \eqref{expansionHc1two}, as established in the works cited above.

\smallskip

Obtaining an analytical solution to the London equation~\eqref{Londoneq} for general domains and normalized applied fields remains a challenging problem. The main objective of this paper is therefore to numerically solve~\eqref{Londoneq} under fairly general hypotheses on both $\Omega$ and $\Hzeroex$. The determination of $\Bzero$, which depends on the solution of the London equation, presents similar analytical difficulties. In view of the key role played by $\Bzero$ in the nucleation of vortex filaments in the magnetic \ac{gl} model of superconductivity, we also aim to compute it numerically.

\subsection{Approximation strategy}

For problems posed on bounded domains, a widely used discretization technique is the \ac{fem}; see \cites{monk2003finite,brenner2008mathematical}. A triangulation of the domain leads to a finite-dimensional approximation in which local polynomial spaces are used as an ansatz for the solution. In the case of Maxwell's equations, this approach is a standard tool in computational electromagnetism; see \cite{hiptmair2002finite}.

\smallskip 

The London equation~\eqref{Londoneq}, however, couples the interior field to an exterior field defined on an unbounded domain. We account for the exterior field without discretizing the exterior domain by representing it through boundary integral operators acting on suitable trace data. This approach is standard for exterior problems in computational electromagnetics, particularly in the context of time-harmonic Maxwell equations; see, for example, \cites{hiptmair2002symmetric,buffa2003boundary,sauter2010boundary}. The resulting discretization scheme is known as the \ac{bem}.

\smallskip 

This boundary integral representation naturally leads to a coupled \ac{fem-bem} formulation of the London equation~\eqref{Londoneq}, viewed as a transmission problem between the bounded interior domain and the unbounded exterior domain. Such a formulation requires the discretization only of the field in the interior domain and of suitable data on its boundary. Previous numerical studies of London-type models for finite superconductors have typically incorporated the vacuum field through Biot--Savart or Green-function integral formulations, especially for thin films and SQUID-type devices \cites{Brandt1994,Brandt1996,KhapaevEtAl2003,ClemBrandt2005}. To the best of our knowledge, a three-dimensional $H(\curl)$-conforming \ac{fem-bem} transmission discretization of the bulk London problem in the full exterior domain has not previously been analyzed in this form.

\smallskip

We next turn to the numerical determination of $\Bzero$, using the field $\Hzero$ computed through the coupled \ac{fem-bem} formulation described above. More precisely, $\Bzero$ is recovered by solving the constrained curl--curl problem~\eqref{eq:B0problem}, with $\Hzero$ as its source term. To handle the divergence-free constraint, we introduce a scalar Lagrange multiplier and formulate~\eqref{eq:B0problem} as a mixed saddle-point problem, following the approach of Kikuchi \cite{fumio1987mixed}. The well-posedness of this formulation follows directly from the de Rham complex associated with the standard differential operators. Furthermore, the theory of \ac{fe} Exterior Calculus \cite{arnold2018finite} ensures that standard conforming \ac{fe} schemes yield stable and convergent discretizations. We therefore adopt this strategy and establish a Strang-type estimate that guarantees the approximability of $\Bzero$ when the numerically computed field $\Hzero$, rather than the exact solution of~\eqref{Londoneq}, is used as the source term in~\eqref{eq:B0problem}.

\smallskip

The overall solution strategy that we propose, based on the sequential computation of $\Hzero$ and $\Bzero$, falls within standard \ac{fem} and \ac{bem} frameworks and therefore allows for an efficient and scalable set of algorithms. This involves the use of preconditioned iterative methods for each linear problem, more specifically GMRES \cite{saad1986gmres} with blockwise \ac{amg} iterations \cite{hackbusch2013multi}. We implement this strategy using the FEniCS library \cite{alnaes2015fenics} and the PETSc interface \cite{balay2019petsc} for the numerical linear algebra. The \ac{bem} block is naturally well-conditioned due to the use of second-kind \acp{bie} with a stable duality pairing in trace spaces \cite{buffa2007dual}. Implementation follows through an extension of the Bempp-cl library \cite{Betcke2021} to the setting of static Maxwell boundary integral operators.

\subsection{Application to the isoflux problem}

We conclude by using the computed field $\Bzero$ to provide numerical evidence of a qualitative change in the isoflux problem. For a ball under a constant normalized applied magnetic field, the unique maximizer is the diameter aligned with the field. We then consider ellipsoids whose major axis is aligned with the applied field, so that the major axis plays the corresponding role. Our computations suggest that this curve ceases to be optimal when the ellipsoid is sufficiently elongated. Indeed, we construct off-axis competitors with a larger isoflux ratio, whose geometry points toward the emergence of U-shaped vortex configurations. These competitors lie within two-dimensional sections of the ellipsoid determined by planes containing the major axis, and remain entirely on one side of it. For a representative ellipsoid, Figure~\ref{fig:ellipsoid} shows the competitor with the largest isoflux ratio within the family considered in our computations, illustrating the U-shaped geometry toward which the maximizing curves appear to evolve.
\begin{figure}[ht!]
    \centering
    \includegraphics[width=1\linewidth]{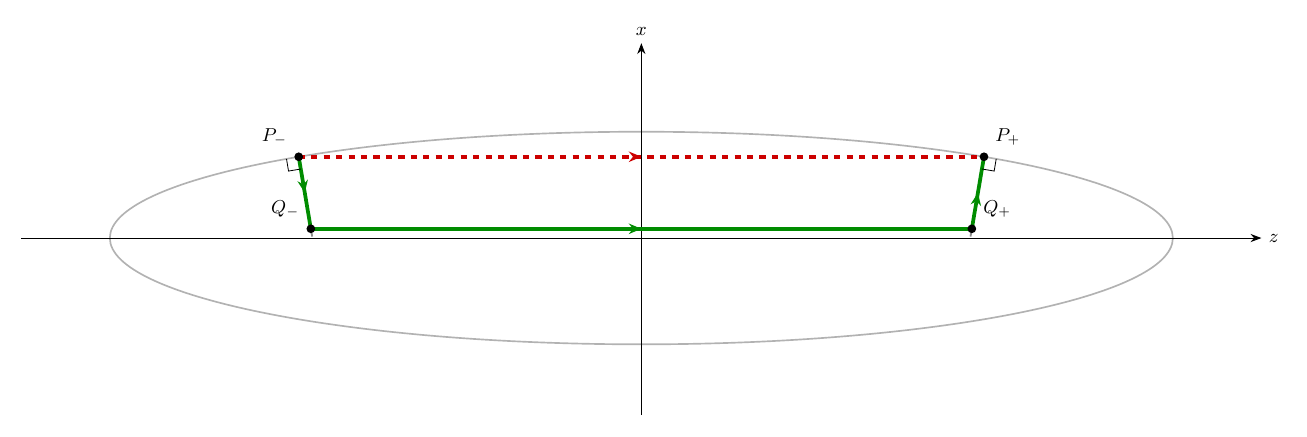}
    \caption{Off-axis competitor attaining the largest isoflux ratio within the family considered, with an isoflux ratio larger than that of the major axis of the ellipsoid and a geometry resembling a U-shaped vortex configuration. See Section~\ref{sec:ellipsoid} for details.}
    \label{fig:ellipsoid}
\end{figure}

\smallskip

The rotational symmetry of the ellipsoid has two distinct consequences. First, rotations of an off-axis configuration generate a continuous family of equivalent configurations, suggesting non-uniqueness in the isoflux problem. Second, the continuous rotational invariance introduces a degenerate direction along this family.

Consequently, the refined analysis of \cites{RSS,RSS2}, where uniqueness and non-degeneracy of the maximizing curve of the isoflux problem are essential assumptions, cannot be applied. This suggests a qualitatively different vortex-nucleation scenario, governed by a degenerate family of competing curves rather than by a single distinguished filament, and points toward the possible emergence of U-shaped configurations analogous to those observed in rotating Bose--Einstein condensates \cites{BEC4,BEC3}; see Section~\ref{sec:ellipsoid} for details.

\subsection*{Plan of the paper} 
The remainder of the paper is organized as follows.  In Section~\ref{sec:derivation}, which can be read independently of the rest of the paper, we derive the London equation from the \ac{gl} model.
Section~\ref{sec:functional} introduces the functional setting and the trace spaces used throughout the paper. In Section~\ref{sec:transmission}, we reformulate the London equation as a transmission problem suitable for \ac{fem-bem} coupling. Section~\ref{sec:numerical} develops the numerical approximation of $\Hzero$ and $\Bzero$ and establishes the convergence of the overall formulation. In Section~\ref{sec:tests}, we present numerical tests assessing the convergence of the proposed method and validate it against the explicit solution available for a ball under a uniform normalized applied field. In Section~\ref{sec:ellipsoid}, we use the computed field $\Bzero$ to explore the isoflux problem in ellipsoidal domains and provide numerical evidence that, for sufficiently elongated geometries, the major axis ceases to be optimal in favor of off-axis competitors resembling U-shaped vortex configurations. The explicit solution used in the numerical validation is presented in Appendix~\ref{sec:solball}, while the layer potentials, boundary integral operators, representation formula, and Calder\'on equations required for the treatment of the exterior problem are collected in Appendix~\ref{sec:BIOs-results}.

\section{Derivation of the London equation}\label{sec:derivation}

In this section, following the presentation in \cite{Rom2}, we recall how the London equation arises from the \ac{gl} model of superconductivity. In an applied magnetic field and at fixed temperature below the critical one, the state of a superconducting sample is modeled by the energy functional
\begin{equation}\label{GLenergy}
    GL_\ep(u,A) \coloneqq \frac12\int_\Omega |\nabla_A u|^2+\frac{1}{2\ep^2}(1-|u|^2)^2+\frac12\int_{\RR^3}|\curl A-\Hex|^2.
\end{equation}
Here, $u:\Omega \rightarrow \mathbb{C}$ is the \emph{order parameter}. Its squared modulus $|u|^2$ represents the local density of Cooper pairs, the paired-electron states that underlie superconductivity in \ac{bcs} theory, and thus encodes the local superconducting state. $A:\RR^3 \rightarrow \RR^3$ is the electromagnetic vector potential generating the induced magnetic field $H \coloneqq \curl A$. $\nabla_A$ denotes the covariant derivative, defined as $\nabla - iA$. 

\smallskip
The \ac{gl} model is a $\mathbb U(1)$ gauge theory, meaning that all physically meaningful quantities are invariant under the gauge transformations 
$$
u\mapsto ue^{i\phi}\quad\mathrm{and}\quad A\mapsto A+\nabla \phi,
$$
where $\phi$ is any sufficiently regular real-valued function. In what follows, we will work on the divergence-free gauge, that is, we assume without loss of generality that $\diver A=0$ in $\RR^3$.

Since $\diver \Hex=0$ in $\RR^3$, in accordance with the nonexistence of magnetic monopoles in Maxwell's theory of electromagnetism, we deduce that there exists a vector potential $\Aex\in H^1_{\loc}(\RR^3,\RR^3)$ such that
$$
\curl \Aex=\Hex\quad\mbox{and}\quad \diver \Aex=0\,\, \mathrm{in}\ \RR^3.
$$
We also define $\Azeroex\coloneqq \hex^{-1}\Aex$.

\smallskip
One of the contributions of \cite{Rom2} is the construction of a precise approximation of the \emph{Meissner state}, that is, the unique (modulo gauge-invariance) solution of the \ac{gl} equations without vortex filaments. This solution corresponds to the global minimizer of \eqref{GLenergy} below the first critical field. The approximating configuration, which is roughly obtained by minimizing \eqref{GLenergy} among configurations such that $|u|=1$, is given by
\begin{equation}\label{meissnerconf}
    (\meisconf),
\end{equation}
where $\Azero$ is the unique minimizer in $[\Azeroex + \hsolhom^1(\RR^3,\RR^3)]$ of the energy functional
\begin{equation*}
    J(A) \coloneqq \frac{1}{2}\int_\Omega |A-\nabla \phiA|^2 + \frac{1}{2}\int_{\RR^3} |\curl(A-\Azeroex)|^2,
\end{equation*}
where $\hsolhom^1(\RR^3,\RR^3)$ denotes the subspace of divergence-free vector fields in the homogeneous Sobolev space $\Dot{H}^1(\RR^3,\RR^3)$, and $\phiA$ is the unique solution in $H^1(\Omega)$ with zero average, that is, $\int_\Omega \phiA = 0$, of the elliptic problem
\begin{equation*}
    \left\{
    \begin{array}{rcll}
        \diver\left(A - \nabla \phiA\right)  &=& 0 &\mathrm{in}\ \Omega\\
        (A-\nabla \phiA) \cdot \nu &=& 0 &\mathrm{on}\ \partial \Omega
    \end{array}
    \right.
    \iff 
    \left\{
    \begin{array}{rcll}
        \Delta\phiA  &=& 0 &\mathrm{in}\ \Omega\\
        \nabla \phiA \cdot \nu &=& A\cdot \nu &\mathrm{on}\ \partial \Omega,
    \end{array}
    \right.
\end{equation*}
where the equivalence comes from $\diver A=0$ in $\RR^3$.
In view of the Helmholtz--Hodge decomposition, the preceding system yields
\begin{equation}\label{b0 existence}
    A-\nabla \phiA = \curl \BA \quad \mbox{in }\Omega
\end{equation}
for a unique vector field $\BA$ satisfying $\diver \BA = 0$ in $\Omega$ and $\BA \times \nu = 0$ on $\partial\Omega$. The phase appearing in \eqref{meissnerconf} is then defined by $\phizero\coloneqq \phiAzero$. We also set $\Bzero=\BAzero$.

\smallskip
The Euler--Lagrange equation solved by the minimizer $\Azero$ of $J$ is given by
\begin{equation}\label{b0 r3 euler lagrange}
    \curl \curl (\Azero - \Azeroex) + \mathbf{1}_\Omega \curl \Bzero = 0\quad \mathrm{in }\ \RR^3.
\end{equation}

Let $\Hzero\coloneqq \curl \Azero$. By taking the curl of \eqref{b0 r3 euler lagrange} in both $\Omega$ and $\RR^3\setminus \overline\Omega$, we obtain the London equation for the magnetic field \eqref{Londoneq}. The condition $\lim_{|x|\to\infty}|\Hzero-\Hzeroex|=0$ follows from the fact that $\Azero\in [\Azeroex + \hsolhom^1(\RR^3,\RR^3)]$.

\smallskip
Let us now derive the boundary condition that $\Hzero$ satisfies on $\partial\Omega$, that is,
\begin{equation}\label{boundarycondH0}
    [\Hzero-\Hzeroex]_{\partial\Omega}=0.
\end{equation}
For this, we let $X$ be an arbitrary smooth vector field with compact support in $\RR^3$. Testing \eqref{b0 r3 euler lagrange} against $X$ and integrating by parts the first term, we find
\begin{equation}\label{bcH1}
\int_{\RR^3} (\Hzero-\Hzeroex)\cdot \curl X+ \int_\Omega \curl \Bzero \cdot X=0.
\end{equation}
Doing the same with $X\mathbf{1}_\Omega$, we obtain
\begin{equation}\label{bcH2}
\int_{\Omega} (\Hzero-\Hzeroex)\cdot \curl X-\int_{\partial\Omega} \left((\Hzero-\Hzeroex)\times \nu \right)\cdot X+ \int_\Omega \curl \Bzero \cdot X=0.
\end{equation}
Analogously, now using $X\mathbf{1}_{\RR^3\setminus \Omega}$ as test vector field, we find
\begin{equation}\label{bcH3}
\int_{\RR^3\setminus \Omega} (\Hzero-\Hzeroex)\cdot \curl X-\int_{\partial(\RR^3\setminus\Omega)}\left((\Hzero-\Hzeroex)\times (-\nu) \right)\cdot X=0.
\end{equation}
Summing \eqref{bcH2} and \eqref{bcH3}, and using \eqref{bcH1}, we deduce that
$$
\int_{\partial\Omega} ( [\Hzero-\Hzeroex]_{\partial\Omega}\times \nu )\cdot X=0,
$$
and hence $[\Hzero-\Hzeroex]_{\partial\Omega}\times \nu=0$. 

\smallskip
On the other hand, we have that 
$$
\diver (\Hzero-\Hzeroex)=0\quad \mbox{in }\RR^3.
$$
Arguing similarly as above, one finds that $[\Hzero-\Hzeroex]_{\partial\Omega}\cdot \nu=0$, and therefore \eqref{boundarycondH0} holds. 

\smallskip
Finally, we observe that \eqref{eq:B0problem} follows by taking the curl of \eqref{b0 r3 euler lagrange} in $\Omega$ and combining the resulting equation with the restriction of \eqref{Londoneq} to $\Omega$.

\section{Functional setting}\label{sec:functional}
We rely on the standard Sobolev space $H^1(\Omega)$ and its natural trace space, denoted by
\begin{equation*}
    H^{1/2}(\partial\Omega)
    \coloneqq
    \gamma_D\bigl(H^1(\Omega)\bigr),
\end{equation*}
where $\gamma_D$ is the Dirichlet trace operator. The dual space of
$H^{1/2}(\partial\Omega)$ is denoted by $H^{-1/2}(\partial\Omega)$.

The vector-valued Sobolev spaces used throughout this work are
\begin{equation*}
    \begin{aligned}
        H(\curl,\Omega)
        &\coloneqq
        \left\{
            U\in[L^2(\Omega)]^3
            \ : \
            \curl U\in[L^2(\Omega)]^3
        \right\},\\
        H(\curl^2,\Omega)
        &\coloneqq
        \left\{
            U\in H(\curl,\Omega)
            \ : \
            \curl U\in H(\curl,\Omega)
        \right\},\\
        H(\diver,\Omega)
        &\coloneqq
        \left\{
            U\in[L^2(\Omega)]^3
            \ : \
            \diver U\in L^2(\Omega)
        \right\}.
    \end{aligned}
\end{equation*}

We define the interior trace operators by
\begin{equation*}
\begin{aligned}
\gammatau^- U
&\coloneqq U\times\nu,
&&
\gammatau^-:
H(\curl,\Omega)
\longrightarrow
H^{-1/2}(\diver_{\partial\Omega},\partial\Omega),\\
\gammaN^- U
&\coloneqq \curl U\times\nu,
&&
\gammaN^-:
H(\curl^2,\Omega)
\longrightarrow
H^{-1/2}(\diver_{\partial\Omega},\partial\Omega),\\
\gamman^- U
&\coloneqq U\cdot\nu,
&&
\gamman^-:
H(\diver,\Omega)
\longrightarrow
H^{-1/2}(\partial\Omega),
\end{aligned}
\end{equation*}
where
\begin{equation*}
    H^{-1/2}(\diver_{\partial\Omega},\partial\Omega)
    \coloneqq
    \left\{
        \lambda\in H^{-1/2}_{\times}(\partial\Omega)
        \ : \
        \diver_{\partial\Omega}\lambda
        \in H^{-1/2}(\partial\Omega)
    \right\}
\end{equation*}
is the natural tangential Sobolev space, as defined in
\cite{buffa2002traces}, and $\diver_{\partial\Omega}$ is the surface divergence operator. Here, we write $\mathbf{H}^{-1/2}_{\times}(\partial\Omega)$ for the dual space of 
$$\mathbf{H}^{1/2}_{\times}(\partial\Omega) \coloneqq \gammatau ([H^1(\Omega)]^3),$$ 
as in \cite{buffa2002traces}*{Section~2} and \cite{buffa2003boundary}*{Definition~1}.

The exterior trace operators
$\gammatau^+$, $\gammaN^+$, and $\gamman^+$
are defined analogously on
$H(\curl,\mathbb{R}^3\setminus\overline{\Omega})$,
$H(\curl^2,\mathbb{R}^3\setminus\overline{\Omega})$, and
$H(\diver,\mathbb{R}^3\setminus\overline{\Omega})$, respectively.

For a vector field $U$ for which the corresponding interior and exterior
traces are well-defined, we denote the jumps across $\partial\Omega$ by
$[\gammatau U]$, $[\gammaN U]$, and $[\gamman U]$.

\section{Transmission formulation of the London equation}
\label{sec:transmission}

To reformulate the London equation~\eqref{Londoneq} in a form suitable for the coupled \ac{fem-bem} formulation developed below, we introduce an auxiliary vector field $\Uzero$. In the interior domain, the purpose of this reformulation is to work directly with the divergence-free component of $\Azero$ arising from the Helmholtz--Hodge decomposition. Indeed, by~\eqref{b0 existence},
\[
\Azero-\nabla\phizero=\curl\Bzero
\quad\mbox{in }\Omega.
\]
Thus, replacing $\Azero$ by $\Azero-\nabla\phizero$ allows us to write the interior equation directly in terms of $\curl\Bzero$, without explicitly retaining the gradient component $\nabla\phizero$.

In the exterior domain, we instead start from $\Azero-\Azeroex$ and introduce an additional gradient correction. More precisely, we define $\psi$ as the unique harmonic extension of $\phizero|_{\partial\Omega}$ to the exterior domain that vanishes at infinity:
\begin{equation*}
\left\{
\begin{array}{rcll}
\Delta\psi&=&0&\mathrm{in}\ \RR^3\setminus\overline{\Omega},\\
\psi&=&\phizero&\mathrm{on}\ \partial\Omega,\\
\displaystyle\lim_{|x|\to\infty}\psi(x)&=&0.&
\end{array}
\right.
\end{equation*}
In particular, $\nabla\psi\in L^2(\RR^3\setminus\overline{\Omega})$. Since $\psi$ and $\phizero$ have the same trace on $\partial\Omega$, their tangential derivatives agree there, and hence
\begin{equation}\label{eq:potential-tangential-traces}
(\nabla\phizero-\nabla\psi)\times\nu=0
\quad\mathrm{on}\ \partial\Omega.
\end{equation}

We then define
\begin{equation}\label{eq:Uzero-definition}
\Uzero\coloneqq
\begin{cases}
\Azero-\nabla\phizero&\text{in }\Omega,\\
\Azero-\Azeroex-\nabla\psi&\text{in }\RR^3\setminus\overline{\Omega}.
\end{cases}
\end{equation}
The exterior component of $\Uzero$ decays at infinity. Moreover, since $\curl\nabla\psi=0$, the gradient correction does not alter the relation between $\curl\Uzero$ and the magnetic field or the exterior equation satisfied by $\Uzero$.
In fact, the magnetic field $\Hzero$ is recovered from $\Uzero$ through
\begin{equation}\label{eq:Hzero-from-Uzero}
\curl\Uzero=\Hzero
\quad\mbox{in }\Omega,
\qquad
\curl\Uzero=\Hzero-\Hzeroex
\quad\mbox{in }\RR^3\setminus\overline{\Omega}.
\end{equation}

The asymmetric definition of $\Uzero$ in~\eqref{eq:Uzero-definition} is deliberate. In the interior domain, we do not subtract $\Azeroex$, as doing so would require incorporating the Helmholtz--Hodge decomposition of this additional term into the formulation. With the present choice, the interior equation has the right-hand side
\begin{equation*}
\mathbf{f}_{\mathrm{ex}}\coloneqq\curl\curl\Azeroex=\curl\Hzeroex,
\end{equation*}
which vanishes, in particular, when the normalized applied magnetic field is constant. In the exterior domain, by contrast, subtracting $\Azeroex$ produces a homogeneous equation. Using the Euler--Lagrange equation~\eqref{b0 r3 euler lagrange}, we therefore obtain
\begin{equation*}
\curl\curl\Uzero=0
\quad\text{in }\RR^3\setminus\overline{\Omega},
\qquad
\curl\curl\Uzero+\Uzero=\mathbf{f}_{\mathrm{ex}}
\quad\text{in }\Omega.
\end{equation*}

It remains to determine the transmission conditions on $\partial\Omega$. From~\eqref{eq:potential-tangential-traces}, we obtain
\begin{equation*}
\begin{aligned}
[\gammatau\Uzero]
&=(\Azero-\Azeroex-\nabla\psi)\times\nu-(\Azero-\nabla\phizero)\times\nu\\
&=-\gammatau\Azeroex
\quad\mathrm{on}\ \partial\Omega.
\end{aligned}
\end{equation*}
Moreover, \eqref{eq:Hzero-from-Uzero} and $[\Hzero-\Hzeroex]_{\partial\Omega}=0$ give
\begin{equation*}
\begin{aligned}
[\gammaN\Uzero]
&=((\Hzero-\Hzeroex)-\Hzero)\times\nu\\
&=-\Hzeroex\times\nu\\
&=-\gammaN\Azeroex
\quad\mathrm{on}\ \partial\Omega.
\end{aligned}
\end{equation*}
Finally, the definition of $\phizero$ gives
\begin{equation*}
\gamman^-\Uzero=(\Azero-\nabla\phizero)\cdot\nu=0
\quad\mathrm{on}\ \partial\Omega,
\end{equation*}
whereas
\begin{equation*}
\gamman^+\Uzero=(\Azero-\Azeroex-\nabla\psi)\cdot\nu
\quad\mathrm{on}\ \partial\Omega.
\end{equation*}
Therefore,
\begin{equation}\label{eq:Uzero-normal-jump}
[\gamman\Uzero]=(\Azero-\Azeroex-\nabla\psi)\cdot\nu
\quad\mathrm{on}\ \partial\Omega.
\end{equation}

Collecting the preceding identities, we conclude that $\Uzero$ satisfies the transmission problem
\begin{equation}\label{eq:Uzero-problem}
\begin{aligned}
\curl\curl\Uzero&=0&&\text{in }\RR^3\setminus\overline{\Omega},\\
\curl\curl\Uzero+\Uzero&=\mathbf{f}_{\mathrm{ex}}&&\text{in }\Omega,\\
[\gammatau\Uzero]&=-\gammatau\Azeroex&&\text{on }\partial\Omega,\\
[\gammaN\Uzero]&=-\gammaN\Azeroex&&\text{on }\partial\Omega,\\
[\gamman\Uzero]&=\gamman^+(\Azero-\Azeroex-\nabla\psi) &&\text{on }\partial\Omega,\\
\lim\limits_{|x|\to\infty}|\Uzero(x)|&=0.&
\end{aligned}
\end{equation}

Although the normal jump~\eqref{eq:Uzero-normal-jump} does not vanish in general, it will not enter the coupled \ac{fem-bem} formulation developed in Section~\ref{sec:numerical}. Indeed, the exterior normal trace appears only through a gradient term in the representation formula for $\Uzero$, which vanishes upon application of the curl trace $\gammaN^+$.

\smallskip 
Thus, the computation of $\Hzero$ reduces to solving the transmission problem~\eqref{eq:Uzero-problem} for $\Uzero$ and subsequently recovering $\Hzero$ from~\eqref{eq:Hzero-from-Uzero}. Once $\Hzero$ has been determined, $\Bzero$ is obtained by solving the constrained curl--curl problem~\eqref{eq:B0problem}, with $\Hzero$ as its source term.

To obtain a well-posed weak formulation of \eqref{eq:B0problem}, we use the Kikuchi formulation \cite{fumio1987mixed} which relies on the homogeneous boundary conditions spaces given by 
        $$
            \begin{aligned}
                \HzeroCurl &\coloneqq \left\{H\in H(\curl,\Omega): H\times \nu = 0\text{ on $\partial\Omega$}\right\}, \\
                \Hzerone &\coloneqq \left\{p\in H^1(\Omega): p = 0\text{ on $\partial\Omega$}\right\},
            \end{aligned}
        $$
    where we have chosen a circle above the $H$ and not a zero below (the more common choice) to avoid confusion with the variable $\Hzero$. The Kikuchi formulation of Maxwell's equations thus reads: Find $\Bzero \in \HzeroCurl$ and $p \in \Hzerone$ such that
\begin{equation}\label{eq:B weak}
  \begin{aligned}
      (\curl \Bzero, \curl B^*) + (\grad p, B^*) &= (\Hzero, B^*) && \forall B^*\in
\HzeroCurl, \\
      (\grad p^*, \Bzero) &=0 && \forall p^* \in \Hzerone.
  \end{aligned}
\end{equation}
This problem is equivalent to \eqref{eq:B0problem}. Its well-posedness follows from \cite{boffi2013mixed}*{Theorem~11.2.1}.
\begin{proposition}[{\cite{boffi2013mixed}*{Theorem~11.2.1}}]
  The bilinear form $b\colon \HzeroCurl\times \Hzerone\to\mathbb{R}$,
  defined by $b(B,p)\coloneqq (\grad p, B)$, satisfies the inf-sup condition
  $$\inf_{p\in \Hzerone\setminus\{0\}}\sup_{B\in \HzeroCurl\setminus\{0\}}
  \frac{b(B,p)}{\|B\|_{H(\curl,\Omega)}\|p\|_{H^1(\Omega)}} \geq \beta > 0.$$
  \begin{proof}
      For any $p\in \Hzerone$, set $B = \grad p$. Since $p|_{\partial\Omega}=0$,
      we have $\nabla_{\partial\Omega} p = 0$ on $\partial\Omega$, and hence
      $\grad p\times\nu = \nabla_{\partial\Omega} p\times\nu = 0$, so that
      $B\in \HzeroCurl$. Moreover, $\curl\grad p = 0$, and therefore
      $\|\grad p\|_{H(\curl,\Omega)} = \|\grad p\|_{L^2(\Omega)}$. Then
      $$\sup_{B\in \HzeroCurl}\frac{b(B,p)}{\|B\|_{H(\curl,\Omega)}}
      \geq \frac{b(\grad p,\,p)}{\|\grad p\|_{H(\curl,\Omega)}}
      = \|\grad p\|_{L^2(\Omega)} \geq c\|p\|_{H^1(\Omega)},$$
      where the last inequality follows from Poincar\'e's inequality.
  \end{proof}
\end{proposition}
Furthermore, the bilinear form $a(B, B^*) \coloneqq (\curl B, \curl B^*)$
is coercive on $\{B \in \HzeroCurl : \diver B = 0\}$. Indeed,
\cite{monk2003finite}*{Corollary~4.8} gives
$$\|B\|_{L^2(\Omega)} \leq C\|\curl B\|_{L^2(\Omega)}$$
for all $B$ in this space, so that
$$a(B,B) = \|\curl B\|^2_{L^2(\Omega)} \geq \frac{1}{1+C^2}\|B\|^2_{H(\curl,\Omega)}.$$
The well-posedness of \eqref{eq:B weak} then follows from the Babu\v{s}ka--Brezzi theory \cite{boffi2013mixed}*{Theorem~11.2.1}.

\section{Numerical Approximation of the London Equation}\label{sec:numerical}
In this section, we provide a numerical framework that yields a provably convergent
approximation of the magnetic field $\Hzero$ and the vector field $\Bzero$. We
formulate Galerkin schemes for \eqref{eq:Uzero-problem} and \eqref{eq:B weak}.
The \ac{fe} approximation of the solutions will be computed on a regular family
of tetrahedral meshes $\mathcal{T}_h$ of $\Omega$, indexed by the mesh size $h$,
satisfying $\Omega_h = \bigcup_{K\in\mathcal{T}_h} K$. The induced triangulation
on $\partial\Omega_h$, consisting of the triangular faces of elements in
$\mathcal{T}_h$ lying on the boundary, is denoted $\mathcal{F}_h$.

Given a polynomial degree $k\geq 1$, on $\mathcal{T}_h$ we consider the Lagrange
space \cite{monk2003finite}*{Section~5.2}
$$\P_k(\mathcal{T}_h)\subset H^1(\Omega)$$
of continuous piecewise polynomials of degree $k$, and the N\'ed\'elec space
\cite{monk2003finite}*{Section~5.5}
$$\NED_k(\mathcal{T}_h)\subset H(\curl,\Omega)$$
of $H(\curl)$ conforming elements with tangential continuity. On the surface mesh
$\mathcal{F}_h$, we consider the Rao--Wilton--Glisson space
$$\mathrm{RWG}(\mathcal{F}_h)\subset
H^{-1/2}(\diver_{\partial\Omega},\partial\Omega).$$
We further introduce a dual barycentric mesh $\widetilde{\mathcal{F}}_h$ of $\partial\Omega_h$,
on which we define the Buffa--Christiansen space \cite{buffa2007dual}
$$\mathrm{BC}(\widetilde{\mathcal{F}}_h)\subset
H^{-1/2}(\diver_{\partial\Omega},\partial\Omega),$$
which serves as the test space in the stable duality pairing
$\langle\cdot,\cdot\rangle_\times$ between
$H^{-1/2}(\diver_{\partial\Omega},\partial\Omega)$ and itself.
\subsection{Approximation of the magnetic field \texorpdfstring{$\Hzero$}{H\textzeroinferior}}

In order to discretize \eqref{eq:Uzero-problem} and solve for $\Hzero$, we aim
to find $\Uzero \in H(\curl, \mathbb{R}^3\setminus \partial \Omega)$ satisfying the interior variational equation 
\begin{equation*}
  \begin{aligned}
      (\curl \Uzero, \curl \Uzero^*) + (\Uzero, \Uzero^*) - \langle \gammaN^- \Uzero,
\gammatau^- \Uzero^*\rangle_{\partial\Omega} &= (\mathbf{f}_{\mathrm{ex}}, \Uzero^*) && \forall
\Uzero^*\in H(\curl,\Omega),
  \end{aligned}
\end{equation*}
and the jump conditions of \eqref{eq:Uzero-problem}
\begin{equation}\label{eq:jump-U}
\begin{aligned}
  [ \gammatau \Uzero ] &= -\gammatau \Azeroex && \text{ on }\partial \Omega,\\
  [ \gammaN \Uzero ]   &= -\gammaN \Azeroex   && \text{ on }\partial \Omega,\\
  [ \gamman \Uzero ]   &= \gamman^+(\Azero-\Azeroex-\nabla\psi)   && \text{ on }\partial \Omega.
\end{aligned}
\end{equation}

For the exterior problem, we use a representation of $\Uzero$ in $ \RR^3 \setminus \overline{\Omega}$ based on integral operators and the fundamental solution of the Laplace problem (see Proposition \ref{prop:rep-formula} in Appendix \ref{sec:BIOs-results}). The representation is given by
    \begin{equation}\label{eq:rep-formula} 
    \begin{aligned}
       \Uzero(x) &= \mathbf{\Psi}_{\mathrm{DL}}(\gammatau^+ \Uzero)(x) - \mathbf{\Psi}_{\mathrm{SL}}(\gammaN^+ \Uzero)(x) + \nabla \Psi_{\mathrm{SL}}(\gamman^+ \Uzero)(x), \qquad \text{ for }x\in \mathbb{R}^3\setminus \overline{\Omega}. 
        \end{aligned}
    \end{equation}
Note that in \eqref{eq:rep-formula} we explicitly observe that a non vanishing exterior normal trace of $\Uzero$ appears as a gradient contribution in the exterior field.
We can obtain boundary integral equations for $\gammatau^+\Uzero $ and $\gammaN^+ \Uzero$ by applying trace operators to \eqref{eq:rep-formula}. Since the curl of a gradient is always zero, by applying $\gamma^+_N$ to \eqref{eq:rep-formula}, we obtain 
\begin{equation}\label{eq:calderon2}
\begin{aligned} 
    -\mathbf{W}_0(\gammatau^+\Uzero) +\left(\tfrac{1}{2}\mathbf{I} + \mathbf{K}_0 \right)(\gammaN^+ \Uzero) = 0,
\end{aligned}
\end{equation}
where we used boundary integral operators as defined in \eqref{eq:BIOs} in Appendix~\ref{sec:BIOs-results}. It becomes clear that we only need information on tangential traces to solve our problem of interest.
These are bounded operators from $H^{-1/2}(\diver_{\partial\Omega}, \partial\Omega) $ to itself.
Using \eqref{eq:calderon2}, we can define a continuous version of a Steklov--Poincaré operator.

The duality pairing that makes $H^{-1/2}(\diver_{\partial\Omega}, \partial\Omega)$ its own dual involves a rotation: for $\lambda, \lambda^*\in H^{-1/2}(\diver_{\partial\Omega}, \partial\Omega),$ we define
\begin{equation*}
    \langle \lambda, \lambda^*\rangle_{\times} \coloneqq \langle \lambda, \nu \times \lambda^*\rangle.
\end{equation*}
We denote $\eta \coloneqq \gammatau^- \Uzero$ and $\lambda \coloneqq \gammaN^- \Uzero. $ Then \eqref{eq:calderon2} combined with the jump conditions \eqref{eq:jump-U} leads to
\begin{equation*}
\begin{aligned} 
    -\mathbf{W}_0\eta +\left(\tfrac{1}{2}\mathbf{I} + \mathbf{K}_0 \right)\lambda = \mathbf{W}_0(\gammatau\Azeroex) -\left(\tfrac{1}{2}\mathbf{I} + \mathbf{K}_0 \right)(\gammaN \Azeroex) \eqqcolon \lambda_{0,\ex}.
\end{aligned}
\end{equation*}
A weak formulation for \eqref{eq:calderon2} reads as follows: we seek $\lambda \in H^{-1/2}(\diver_{\partial\Omega}0, \partial\Omega)$ such that
\begin{equation*}
     -\langle\mathbf{W}_0 \eta , \lambda^*\rangle_{\times} + \langle \left(\tfrac{1}{2}\mathbf{I} + \mathbf{K}_0 \right)\lambda, \lambda^*\rangle_{\times} = \langle \lambda_{0, \ex}, \lambda^*\rangle_\times \qquad \forall\lambda^*\in H^{-1/2}(\diver_{\partial\Omega}0, \partial\Omega),
\end{equation*}
where $H^{-1/2}(\diver_{\partial\Omega}0, \partial\Omega)$ is defined in Appendix~\ref{sec:BIOs-results}, eq.~\eqref{eq:div0-trace-space}.
\begin{lemma}
    The operator $\mathbf{K}_0 : H^{-1/2}(\diver_{\partial\Omega}, \partial\Omega) \rightarrow H^{-1/2}(\diver_{\partial\Omega}, \partial\Omega)$ is compact.
\end{lemma}
We are interested in the coupling: find $(U, \lambda) \in H(\curl, \Omega)\times H^{-1/2}(\diver_{\partial\Omega}0, \partial\Omega) $ such that
\begin{equation}\label{eq:U-coupling}
        \begin{aligned}
            (\curl \Uzero, \curl \Uzero^*) + (\Uzero, \Uzero^*) - \langle \lambda, \gammatau^- \Uzero^*\rangle_{\times}&= (\mathbf{f}_{\mathrm{ex}}, \Uzero^*)  && \forall \Uzero^*\in H(\curl,\Omega),\\
            -  \langle \mathbf{W}_0 \gammatau^- \Uzero, \lambda^*\rangle_{\times}+ \langle \left(\tfrac{1}{2}\mathbf{I} + \mathbf{K}_0 \right)\lambda, \lambda^*\rangle_{\times} &= \langle \lambda_{0,\ex}, \lambda^*\rangle && \forall \lambda^* \in H^{-1/2}(\diver_{\partial\Omega}, \partial\Omega).
        \end{aligned}
    \end{equation}

\begin{definition}[Steklov--Poincaré operator]
The operator $$\mathbf{S}_0 \colon H^{-1/2}(\diver_{\partial\Omega},\partial\Omega)
\to H^{-1/2}(\diver_{\partial\Omega}0,\partial\Omega)$$ is defined by
$$
\mathbf{S}_0 \eta \coloneqq \big(\tfrac12\mathbf I+\mathbf K_0\big)^{-1}\mathbf W_0\,\eta,
\qquad \eta \in H^{-1/2}(\diver_{\partial\Omega},\partial\Omega),
$$
equivalently characterized, for $\eta\in H^{-1/2}(\diver_{\partial\Omega},\partial\Omega)$,
as the unique $\mu \coloneqq \mathbf S_0\eta$ satisfying
$$
\big\langle (\tfrac12\mathbf{I}+\mathbf{K}_0)\mu,\ \mu^*\big\rangle_\times
= \langle \mathbf W_0\eta,\ \mu^*\rangle_\times
\qquad \forall\, \mu^* \in H^{-1/2}(\diver_{\partial\Omega}0,\partial\Omega).
$$
\end{definition}

The reduced, one-field bilinear form
$q\colon H(\curl,\Omega)\times H(\curl,\Omega)\to\mathbb R$ associated
with \eqref{eq:U-coupling} is defined by
$$
q(\Uzero,\Uzero^*) \coloneqq
(\curl \Uzero,\curl \Uzero^*) + (\Uzero,\Uzero^*)
- \big\langle \mathbf S_0\,\gammatau^-\Uzero,\ \gammatau^-\Uzero^*\big\rangle_\times,
\qquad \Uzero,\Uzero^*\in H(\curl,\Omega).
$$

\begin{proposition}[Inf-sup condition for the reduced coupling]
Let $\Omega\subset\mathbb R^3$ be a bounded domain with smooth connected
boundary $\partial\Omega$. The bilinear form $q$ satisfies the
inf-sup condition
$$
\inf_{\substack{\Uzero\in H(\curl,\Omega)\\\Uzero\ne0}}\
\sup_{\substack{\Uzero^*\in H(\curl,\Omega)\\ \Uzero^*\ne0}}\
\frac{q(\Uzero,\Uzero^*)}{\|\Uzero\|_{H(\curl,\Omega)}\,\|\Uzero^*\|_{H(\curl,\Omega)}}
\;\ge\;\beta\;>\;0.
$$
\end{proposition}
\begin{proof}
Since $\mathbf{K}_0$ is a compact operator from $H^{-1/2}(\diver_{\partial\Omega}, \partial\Omega)$ to itself, the operator $\mathbf{S}_0$ has a principal part corresponding to $2\mathbf{W}_0$. Then, it is enough to show that the bilinear form
$$
q_0(\Uzero,\Uzero^*) \coloneqq
(\curl \Uzero,\curl \Uzero^*) + (\Uzero,\Uzero^*)
- \big\langle 2\mathbf W_0\,\gammatau^-\Uzero,\ \gammatau^-\Uzero^*\big\rangle_\times,
\qquad \Uzero,\Uzero^*\in H(\curl,\Omega).
$$
satisfies an inf-sup condition with a constant $\beta_0.$ Choose a test function $\Uzero^* = \Uzero + V_0\in H(\curl,\Omega).$ Let $V_0$ be the solution of
\begin{equation}\label{eq:V0-aux}
    (\curl V_0, V_0^*) + (V_0, V_0^*) = \langle 2\mathbf{W}_0\gammatau^-\Uzero, V_0^*\rangle_\times \qquad \forall V_0^*\in H(\curl,\Omega).
\end{equation}
Then, from \eqref{eq:V0-aux}, we can rewrite
\begin{equation*}
    \begin{aligned}
        q_0(\Uzero,\Uzero +V_0) &= \|U_0\|_{H(\curl,\Omega)}^2 + (\curl\Uzero, \curl V_0) + (\Uzero, V_0) \\ 
        &- \big\langle 2\mathbf W_0\,\gammatau^-\Uzero,\ \gammatau^-\Uzero\big\rangle_\times -\big\langle 2\mathbf W_0\,\gammatau^-\Uzero,\ \gammatau^-V_0\big\rangle_\times
    \end{aligned}
\end{equation*}
into
$$
q_0(\Uzero,\Uzero +V_0) = \|\Uzero\|_{H(\curl,\Omega)}^2 - \|V_0\|_{H(\curl,\Omega)}^2.
$$
Hence, it remains to show that $\|V_0\|_{H(\curl,\Omega)}\leq c \|\Uzero\|_{H(\curl,\Omega)}<\|\Uzero\|_{H(\curl,\Omega)}$ for some $c\in(0,1)$, which follows from a standard contraction argument involving the operator $\mathbf{K}_0$; see \cites{elasmi2021johnson,steinbach2011note}.
\end{proof}

In order to have a stable discrete duality pairing $\langle \cdot, \cdot\rangle_\times$ in $H^{-1/2}(\diver_{\partial\Omega}, \partial\Omega)$, we rely on the pairing of $\mathrm{RWG}(\mathcal{F}_h)$ (for trial) and $\mathrm{BC}(\widetilde{\mathcal{F}}_h)$  (for test) boundary element spaces. Therefore, we solve for $(U_{0, h}, \lambda_h)\in \mathrm{NED}(\mathcal T_h)\times \mathrm{RWG}(\mathcal F_h)$ such that
\begin{equation*}
        \begin{aligned}
            (\curl U_{0, h}, \curl U_{0, h}^*) + (U_{0, h}, U_{0, h}^*) - \langle \lambda_h, \gamma_{\tau}^- U_{0, h}^*\rangle_{\times}&= (\mathbf{f}_{\ex}, U^*_{0,h}) && \forall U_{0, h}^*\in \mathrm{NED}_k(\mathcal T_h),\\
            -  \langle \mathbf{W}_0 \gamma_\tau^- U_{0, h}, \lambda_h^*\rangle_{\times}+ \langle \left(\tfrac{1}{2}\mathbf{I} + \mathbf{K}_0 \right)\lambda_h, \lambda_h^*\rangle_{\times} &= \langle \lambda_\ex, \lambda_h^*\rangle && \forall \lambda_h^* \in \mathrm{BC}(\widetilde{\mathcal{F}}_h).
        \end{aligned}
    \end{equation*}

\subsection{Approximation of the vector field \texorpdfstring{$\Bzero$}{B\textzeroinferior}}
Consider the conforming \ac{fe} spaces
    \begin{align*}
        V_h^B &\coloneqq \NED_k(\mathcal T_h) \cap \HzeroCurl, \\
        V_h^p &\coloneqq \P_k(\mathcal T_h) \cap \Hzerone.
    \end{align*}
The discrete Galerkin formulation then looks for $(\Bh, p_h)$ in $V_h^B \times V_h^p$ such that
    \begin{equation}\label{eq:Bh}
        \begin{aligned}
            (\curl \Bh, \curl \Bh^*) + (\grad p_h, \Bh^*) &= \langle \Hzero, \Bh^*\rangle && \forall \Bh^*\in V_h^B, \\
            (\grad p_h^*, \Bh) &=0 && \forall p_h^* \in V_h^p.
        \end{aligned}
    \end{equation}
\begin{lemma}\label{lemma:kikuchi wp}
    Problem \eqref{eq:Bh} is well-posed. 
    \begin{proof}
        This result is shown in \cite{boffi2013mixed}*{Theorem 11.2.2} and its proof relies on the Babu\v ska--Brezzi theory. Indeed, the discrete spaces have been chosen such that $\grad V_h^p \subset V_h^B$, which immediately yields the discrete inf-sup stability of the bilinear form $(\grad p_h^*, \Bh)$. Ellipticity of the positive part $(\curl \Bh, \curl \Bh^*)$ is a consequence of the discrete Friedrichs inequality.
    \end{proof}
\end{lemma}
The Ladysenskaya--Babu\v{s}ka--Brezzi theory establishes the convergence of the proposed \ac{fe} scheme for \eqref{eq:Bh}, which yields the following result: 

\begin{theorem}
    If $\Bzero$ belongs to $H^s(\curl, \Omega)$ with $1/2 <  s \leq k$, then the solutions $\Bh$ and $\Bzero$ satisfy the following error estimate: 
    $$ \| \Bzero - \Bh \|_{H_0(\curl,\Omega)} \leq C\| \Bzero \|_{H^s(\curl,\Omega)} h^s,$$
    for some generic mesh-independent constant $C$.
    \begin{proof}
        This is a standard consequence of the Ceá estimate satisfied by the solution of \eqref{eq:Bh} as in \cite{boffi2013mixed}*{Corollary 11.2.1} and the convergence properties of the Nédélec interpolator \cite{monk2003finite}*{Theorem 5.41}.
    \end{proof}
\end{theorem}

\subsection{Convergence of the overall formulation}
Given that each sub-problem (\eqref{eq:Uzero-problem} and \eqref{eq:B weak}) enjoys standard optimal convergence rates (under sufficient regularity), the final concern regards the convergence of the entire procedure, i.e. the computations
    \begin{enumerate}
        \item of the magnetic field $\Hh$,
        \item and of the vector field $\Bh \coloneqq \Bh(\Hh)$. 
    \end{enumerate}
We now proceed to derive a Strang estimate, which directly establishes the results we are looking for.

\subsection{Strang estimate for \texorpdfstring{$\Bh$}{B\textzeroinferior,ₕ}}
We modify the formulation of \eqref{eq:Bh} to have a mesh-dependent right hand side by replacing $\Hzero$ with $\Hh$. Then, the errors $e_h^B\coloneqq \Bzero - \Bh$, $e_h^H\coloneqq \Hzero-\Hh$, and $e_h^p\coloneqq p - p_h$ satisfy the equation
    \begin{equation}\label{eq:Bh strang error}
        \begin{aligned}
            (\curl e_h^B, \curl \Bh^*) + (\grad e_h^p, \Bh^*) &= (\curl e_h^H, \curl \Bh^*) && \forall \Bh^* \in V_h^B \\
            (\grad p_h^*, e_h^B)  &= 0 && \forall p_h^* \in V_h^p.
        \end{aligned}
    \end{equation}
This yields the following convergence result. 
\begin{lemma}
    If the error in the magnetic field converges as $\| e_h^H \|_{\curl}\leq C h^k$, then the following convergence estimate holds: 
    $$ \| \Bzero - \Bh \|_{\curl{}} \leq C h^k  . $$
    \begin{proof}
        Consider $e_h^B$ the solution to \eqref{eq:Bh strang error}, which is guaranteed to exist in virtue of Lemma~\ref{lemma:kikuchi wp}. Then, the discrete Friedrich inequality yields the existence of some constant $c>0$ such that
        $$ c\| e_h^B \|_{\curl{}} \leq \|\curl e_h^B\|_{L^2}. $$
        This, together with the previously used fact that $p_h = p = 0$ \cite{boffi2013mixed} yields that
        $$ c^2\| e_h^B \|^2_{\curl{}} \leq \|\curl e_h^B \|_{L^2}^2 \leq (\curl e_h^H, \curl e_h^B) \leq \| e_h^H \|_{\curl{}} \| e_h^B \|_{\curl{}}, $$
        which establishes the desired result.
    \end{proof}
\end{lemma}

\section{Numerical tests}\label{sec:tests}
In this section we propose numerical tests to validate our convergence theory. These tests are:
    \begin{enumerate}
        \item Validation of the convergence rates computed in Section~\ref{sec:numerical} for the variables $\Hh$ and $\Bh$ using a manufactured-solution approach.
        \item Theoretical validation test, in which we solve for $\Bzero$ on a ball and compare the numerical solution to the analytical one available in Appendix~\ref{sec:solball}.
    \end{enumerate}
\subsection{Numerical convergence}

\subsubsection{$\Uzero$ problem}
In this section, we validate the convergence of problem \eqref{eq:U-coupling} using a manufactured-solution approach. For this, we consider an analytic solution, given by 
    $$ \Uzero(x,y,z) = \begin{bmatrix} e^x\sin(y) \\ e^x\cos(y) \\ 0 \end{bmatrix},$$
in $\Omega$ and zero in the exterior. We compensate the right hand side accordingly for this study in $\Omega$. The results are reported in Table~\ref{tab:U-rates}, where the convergence rates expected from the theory are observed. The definitions used are given by
    \begin{equation*}
    \begin{aligned}
        e^h_{U} \coloneqq \|\Uzero - U_{0,h} \|_{H(\curl,\Omega)}, \qquad  r_{U} \coloneqq \frac{\log \frac{e_{U}^h}{e_{U}^{h'}} }{\log \frac{h}{h'}}.
    \end{aligned}
    \end{equation*}

\begin{table}[ht!]
    \centering
    \begin{tabular}{c|c c}
\toprule DoFs & $e^h_U$ & $r_U$  \\  \midrule
  1043   & 1.66e-01 & --   \\
  1345   & 1.46e-01 & 1.13 \\
  3764   & 1.05e-01 & 0.91 \\
  25362  & 5.54e-02 & 0.98 \\
  48409  & 4.48e-02 & 1.08 \\
  111917 & 3.33e-02 & 0.96 \\
  187186 & 2.80e-02 & 1.02 \\ \bottomrule
    \end{tabular}
    \caption{Convergence test: Manufactured solutions for $\Uzero$.}
    \label{tab:U-rates}
\end{table}

\subsubsection{$\Bzero$ problem} 

In this section, we provide a manufactured solution to problem \eqref{eq:Bh} and study the convergence with respect to it. The analytical solution we consider is 
    \begin{equation*}
        \begin{aligned}
            B(x,y) = \begin{bmatrix}x^2y - y^2 \\ x^2 - xy^2 \end{bmatrix}       
        \end{aligned},
    \end{equation*}
which we impose on the entire boundary as a Dirichlet boundary condition. Note that the analytical pressure is $p=0$. We perform this for a uniform refinement of the mesh and obtain the optimal rates reported in Table~\ref{tab:B-rates}, as expected from the theory. The discrete pressure is zero up to machine precision. Note that, in the table, we have used the following definitions:
    \begin{equation*}
    \begin{aligned}
        e^h_B &\coloneqq \|B - B_h\|_{H(\curl, \Omega)}, &\quad&
        e^h_p \coloneqq \|p_h\|_{H^1(\Omega)},  \\
        r_B &\coloneqq \frac{\log \frac{e_B^h}{e_B^{h'}} }{\log \frac{h}{h'}}, &\quad&
        r_p \coloneqq \frac{\log \frac{e_p^h}{e_p^{h'}} }{\log \frac{h}{h'}}.
    \end{aligned}
    \end{equation*}

\begin{table}[ht!]
    \centering
    \begin{tabular}{c|c c | c c}
\toprule DoFs & $e^h_B$ & $r_B$ & $e^h_p$ & $r_p$  \\ \midrule
25 & 3.00e-01 & -- & 3.09e-16 & -- \\
81 & 1.73e-01 & 0.80 & 1.11e-14 & -5.17 \\
289 & 9.33e-02 & 0.89 & 1.17e-14 & -0.07 \\
1089 & 4.85e-02 & 0.94 & 3.02e-13 & -4.69 \\
4225 & 2.48e-02 & 0.97 & 4.07e-12 & -3.75 \\
16641 & 1.25e-02 & 0.99 & 4.45e-11 & -3.45 \\ \bottomrule
    \end{tabular}
    \caption{Convergence test: Manufactured solutions for $\Bzero$.}
    \label{tab:B-rates}
\end{table}

\subsection{Theoretical validation with an analytic solution}
In this section, we study the convergence of the solution in a curved domain, namely the unit ball, against a well-known analytic solution for $\Hzeroex=(0,0,1)$, with $\Azeroex(x,y,z) = \frac 1 2(y,-x,0)$. The full expression is provided in Appendix \ref{sec:solball}, and we provide the convergence for the problem's variables in Table~\ref{tab:UB-rates}. Note that in the table, we have used the following definitions:
    \begin{equation*}
    \begin{aligned}
        e^h_{\Bzero} &\coloneqq \|\Bzero - B_{0,h} \|_{H(\curl, \Omega)}, &\quad&
        e^h_{\Uzero} \coloneqq \|\Uzero - U_{0,h} \|_{H(\curl,\Omega)},  \\
        r_{\Bzero} &\coloneqq \frac{\log \frac{e_{\Bzero}^h}{e_{\Bzero}^{h'}} }{\log \frac{h}{h'}}, &\quad&
        r_{\Uzero} \coloneqq \frac{\log \frac{e_{\Uzero}^h}{e_{\Uzero}^{h'}} }{\log \frac{h}{h'}}.
    \end{aligned}
    \end{equation*}

\begin{table}[ht!]
    \centering
    \begin{tabular}{c|c c | c c}
\toprule DoFs & $e^h_{\Uzero}$ & $r_{\Uzero}$ & $e^h_{\Bzero}$ & $r_{\Bzero}$  \\ \midrule
1043 & 2.45e-02 & -- & 1.79e-01 & --  \\
1345 & 2.12e-02 & 1.24 & 1.59e-01 & 1.04  \\
3764 & 1.49e-02 & 0.98 & 1.10e-01 & 1.02  \\
25362 & 7.21e-03 & 1.11 & 5.51e-02 & 1.06  \\
48409 & 5.74e-03 & 1.16 & 4.39e-02 & 1.15  \\
111917 & 4.27e-03 & 0.96 & 3.28e-02 & 0.94  \\
187186 & 3.56e-03 & 1.06 & 2.74e-02 & 1.04  \\
\bottomrule
    \end{tabular}
    \caption{Convergence test: Analytic solution for $\Uzero$ and $\Bzero$.}
    \label{tab:UB-rates}
\end{table}

The recovered rates are as expected from the theory. We further show the discrete solution $\Bh$ and the discrete approximation of $\Hh$, given by $\curl U_{0,h}$, in Figure~\ref{fig:sphere-solution}. Note that there are some oscillations present in $\Hh$, mainly towards the boundary, as it is the discrete gradient of a \ac{fe} field.

\begin{figure}
    \newcommand{\wid}{0.9}
    \centering
    \begin{subfigure}{0.45\textwidth}
        \includegraphics[width=\wid\linewidth]{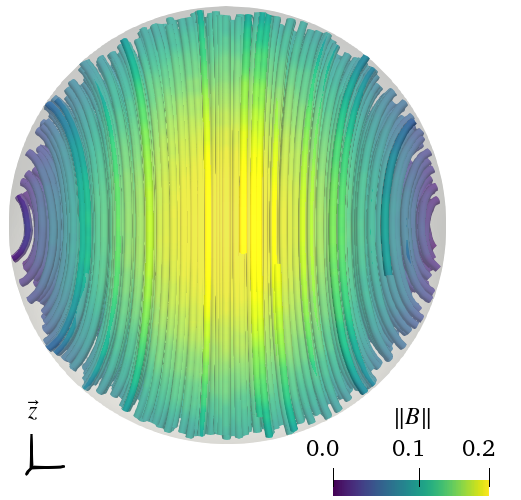}
        \caption{$\Bh$ solution.}
    \end{subfigure}
    \begin{subfigure}{0.45\textwidth}
        \includegraphics[width=\wid\linewidth]{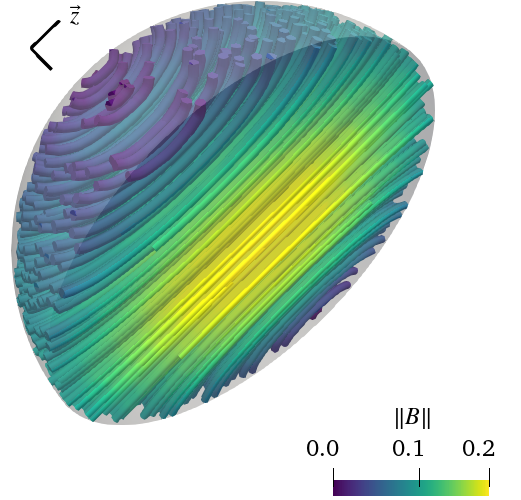}
        \caption{$\Bh$ solution tilted.}
    \end{subfigure}
    
    \begin{subfigure}{0.45\textwidth}
        \includegraphics[width=\wid\linewidth]{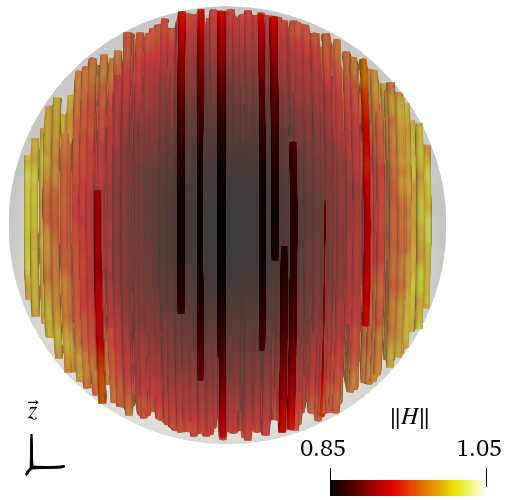}
        \caption{$\Hh$ solution.}
    \end{subfigure}
    \begin{subfigure}{0.45\textwidth}
        \includegraphics[width=\wid\linewidth]{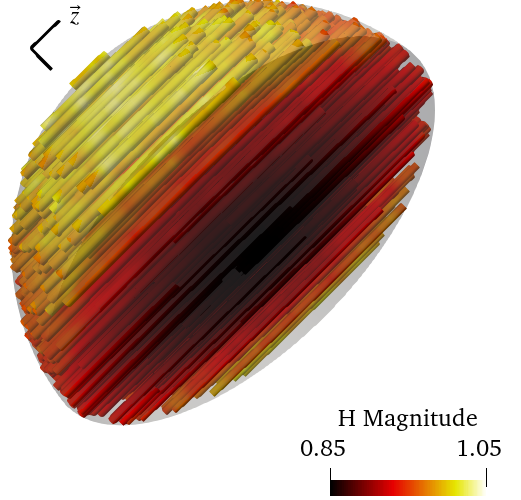}
        \caption{$\Hh$ solution tilted.}
    \end{subfigure}
    \caption{Theoretical validation test on the unit ball.}
    \label{fig:sphere-solution}
\end{figure}

\section{Application to the isoflux problem}\label{sec:ellipsoid}
We consider the prolate ellipsoid
$$
\Omega=\left\{(x,y,z)\in\RR^3:\frac{x^2+y^2}{a^2}+z^2<1\right\},
\qquad 0<a<1,
$$
and assume that the normalized applied magnetic field is
$$
\Hzeroex=\hat z.
$$
Thus, the major axis of the ellipsoid is aligned with the $z$-direction, while the two minor semi-axes, in the $x$- and $y$-directions, have length $a$. Since both the domain and the direction of the applied field are invariant under rotations around the major axis, the corresponding field $\Bzero$ is independent of the azimuthal angle and has no component in the azimuthal direction. As in the analysis of the isoflux problem for the ball carried out in \cites{AlaBroMon,Rom2}, this axisymmetry naturally reduces the problem to curves contained in two-dimensional sections determined by planes containing the major axis and lying entirely on one side of this axis. We work in the section determined by the $(x,z)$-plane and, without loss of generality, consider curves lying on the side $x\geq0$ of the major axis. The boundary of this section is the ellipse
$$
\frac{x^2}{a^2}+z^2=1.
$$

\smallskip

Let $\Gamma$ be an oriented curve contained in this section, lying entirely on one side of the major axis and with endpoints on the boundary of the ellipse. As illustrated in Figure~\ref{fig:vortex-line-sector}, we denote by $\Gamma_{\mathrm{bd}}$ the portion of the boundary of the ellipse joining the endpoints of $\Gamma$, oriented so that $\Gamma\cup\Gamma_{\mathrm{bd}}$ forms a closed curve. We then denote by $S_\Gamma$ the region enclosed by this curve. In particular,
$$
\Gamma_{\mathrm{bd}}=\partial S_\Gamma\cap\partial\Omega.
$$

\begin{figure}[ht!]
\centering
\includegraphics[width=0.62\linewidth]{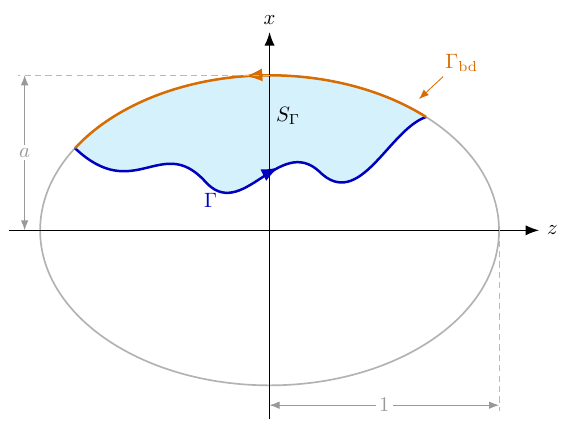}
\caption{An oriented curve $\Gamma$ lying on one side of the major axis and the region $S_\Gamma$ enclosed by $\Gamma$ and the corresponding boundary curve $\Gamma_{\mathrm{bd}}$.}
\label{fig:vortex-line-sector}
\end{figure}

The boundary condition $\Bzero\times\nu=0$ on $\partial\Omega$ is crucial here, since it implies that $\Bzero$ has no tangential component along $\Gamma_{\mathrm{bd}}$ and hence
$$
\int_{\Gamma_{\mathrm{bd}}}\Bzero\cdot d\ell=0.
$$
We may therefore close $\Gamma$ along the boundary without changing its line integral:
$$
\int_\Gamma\Bzero\cdot d\ell=\int_{\partial S_\Gamma}\Bzero\cdot d\ell.
$$
Applying Stokes' theorem in the oriented $(x,z)$-section yields
$$
\int_\Gamma\Bzero\cdot d\ell=\int_{S_\Gamma}\curl\Bzero\cdot\hat y\,dS,
$$
where $\hat y$ is the oriented unit normal to the section. The flux density $\curl\Bzero\cdot\hat y$ is nonnegative; see \cite{RSS}*{Section~3}.

Consequently, the argument used for the ball in \cites{AlaBroMon,Rom2} also applies in the present setting and shows that a curve whose length exceeds that of the major axis cannot be optimal, since its isoflux ratio is smaller than the one attained by the major axis. Moreover, this analysis allows us to restrict our numerical study to curves that are symmetric with respect to the $x$-axis, meet the boundary of the ellipse orthogonally, and have an associated region $S_\Gamma$ that is convex.

\smallskip

Motivated by these geometric properties, we introduce a two-parameter family of piecewise-linear competitors $\Gamma_{(x_0,\lambda)}$. The construction is illustrated in Figure~\ref{fig:isoflux-experiment} and described in detail next. The parameter $x_0$ determines the endpoints of the curve on the boundary of the ellipse, while $\lambda$ controls how far the curve follows the inward normal direction before turning parallel to the major axis.

\begin{figure}[ht!]
\centering
\includegraphics[width=0.7\linewidth]{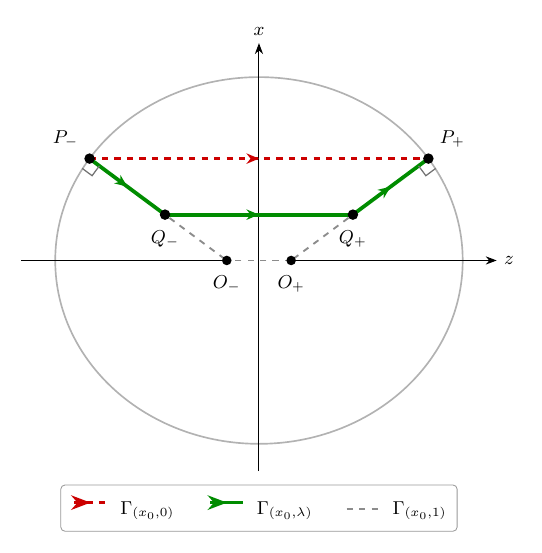}
\caption{Geometric construction of the family $\Gamma_{(x_0,\lambda)}$ in the $(x,z)$-section of the ellipsoid. The figure shows the limiting configuration $\Gamma_{(x_0,0)}$, an intermediate competitor $\Gamma_{(x_0,\lambda)}$, and the configuration $\Gamma_{(x_0,1)}$, which reaches and follows the major axis.}
\label{fig:isoflux-experiment}
\end{figure}

Fix $x_0\in(0,a)$ and define
$$
z_0=\sqrt{1-\left(\frac{x_0}{a}\right)^{\!2}},\qquad
P_\pm=(x_0,\pm z_0)\in\partial\Omega.
$$
For $\lambda\in[0,1]$, set
$$
z_\lambda=z_0(1-\lambda a^2), \qquad
Q_\pm=\bigl(x_0(1-\lambda),\pm z_\lambda\bigr).
$$
For $\lambda\in(0,1]$, the points $Q_\pm$ lie in the interior of $\Omega$, and the segments joining $P_\pm$ to $Q_\pm$ follow the inward normal directions to the ellipse at $P_\pm$. When $\lambda=1$, the points $Q_\pm$ reach the major axis at
$$
O_\pm=\bigl(0,\pm z_0(1-a^2)\bigr).
$$

For $\lambda\in(0,1]$, we define $\Gamma_{(x_0,\lambda)}$ as the piecewise-linear curve connecting, in order,
$$
P_-,\qquad Q_-,\qquad Q_+,\qquad P_+.
$$
These curves are symmetric with respect to the $x$-axis and meet the boundary of the ellipse orthogonally at $P_-$ and $P_+$. We also include the limiting configuration
$$
\Gamma_{(x_0,0)}=[P_-,P_+],
$$
which is parallel to the major axis. Unlike the curves corresponding to $\lambda\in(0,1]$, this limiting curve does not meet the boundary orthogonally; it is included to illustrate the full geometric interpolation. At the other endpoint,
$$
\Gamma_{(x_0,1)}
=
[P_-,O_-]\cup[O_-,O_+]\cup[O_+,P_+],
$$
which reaches and follows the major axis. Intermediate values of $\lambda$ interpolate between these two configurations.

Using the notation for the isoflux ratio presented in the Introduction, every curve in the family satisfies
$$
\R\bigl(\Gamma_{(x_0,\lambda)}\bigr)\leq\R_0.
$$
Although these piecewise-linear curves are not expected to be maximizers themselves, they provide a simple finite-dimensional family of competitors with which to investigate whether the major axis remains optimal.

\smallskip

For comparison with the major axis, we extend the notation by setting
$$
\Gamma_{(0,0)}\coloneqq \left\{(0,0,z)\in\RR^3:-1\leq z\leq1\right\}.
$$
Thus, $\Gamma_{(0,0)}$ denotes the major axis of the ellipsoid.

\smallskip

We numerically solve the London equation and subsequently compute $\Bh$. The resulting fields $\Bh$ and $\Hh$ are shown in Figure~\ref{fig:ellipsoid-solution}, which seem visually comparable to the corresponding results for the ball shown in Figure~\ref{fig:sphere-solution} and present similar oscillations in the discrete curl $\Hh$ towards the boundary.
\begin{figure}[ht!]
    \centering
    \newcommand{\wid}{0.9}
    \begin{subfigure}{0.45\textwidth}
        \includegraphics[width=\wid\linewidth]{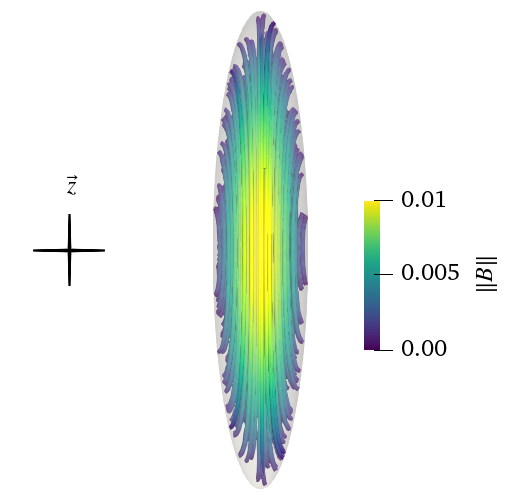}
        \caption{$\Bh$ solution.}
    \end{subfigure}
    \begin{subfigure}{0.45\textwidth}
        \includegraphics[width=\wid\linewidth]{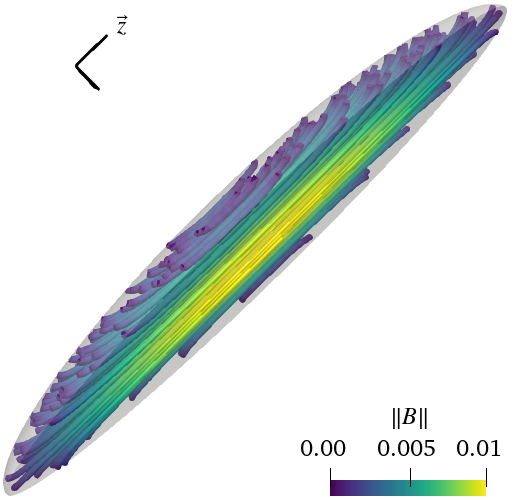}
        \caption{$\Bh$ solution tilted.}
    \end{subfigure}
    
    \begin{subfigure}{0.45\textwidth}
        \includegraphics[width=\wid\linewidth]{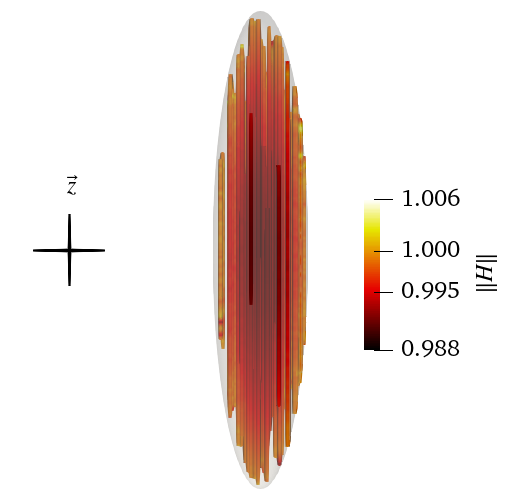}
        \caption{$\Hh$ solution.}
    \end{subfigure}
    \begin{subfigure}{0.45\textwidth}
        \includegraphics[width=\wid\linewidth]{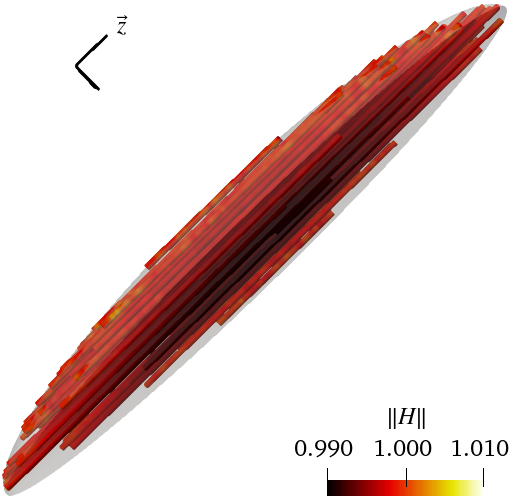}
        \caption{$\Hh$ solution tilted.}
    \end{subfigure}
    \caption{Discrete fields $\Hh$ and $\Bh$ for the ellipsoid, used to evaluate the isoflux ratio.}
    \label{fig:ellipsoid-solution}
\end{figure}

We then evaluate
$$
\R\bigl(\Gamma_{(x_0,\lambda)}\bigr)
$$
for a sample of $x_0$-values in both the unit ball and the ellipsoid with $a=0.2$. The left panel of Figure~\ref{fig:isoflux-ellipsoid} shows that, in the ball case, where $O_+=O_-$ coincides with the origin, none of the curves in the family considered has an isoflux ratio larger than that of the diameter aligned with the applied field, in agreement with its theoretically established optimality. By contrast, the right panel provides clear numerical evidence that the major axis of the elongated ellipsoid is not optimal, as several off-axis competitors attain a larger isoflux ratio.

\begin{figure}[ht!]
    \centering
    \newcommand{\wid}{1}
    \begin{subfigure}{0.45\textwidth}
        \centering
        \includegraphics[width=\wid\linewidth]{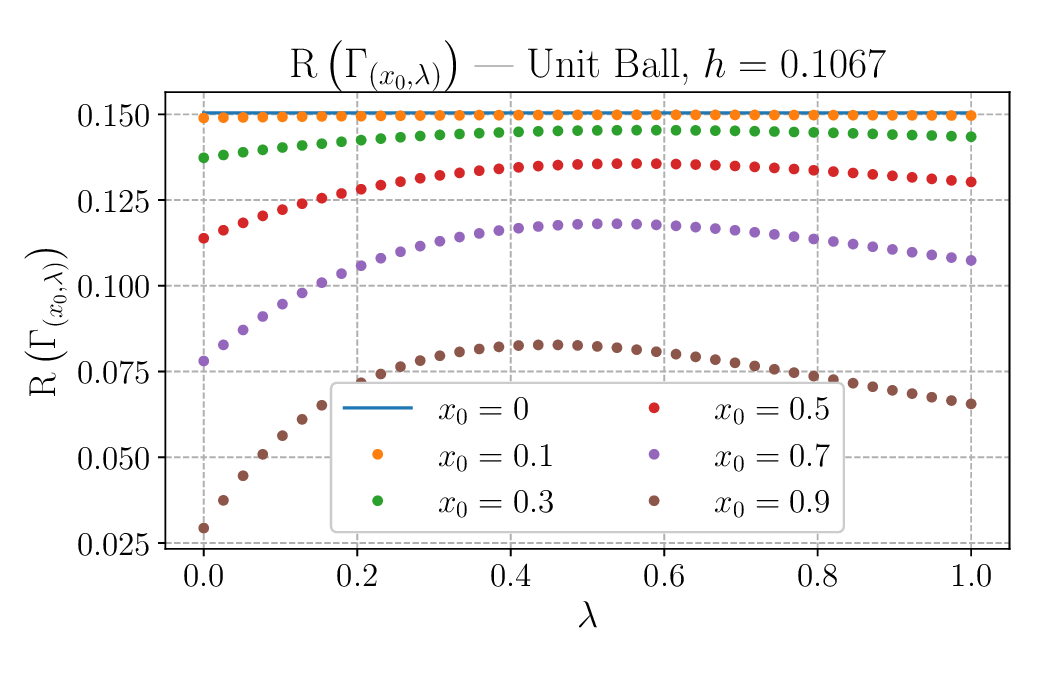}
        \caption{Unit ball.}
    \end{subfigure}
    \begin{subfigure}{0.45\textwidth}
        \centering
        \includegraphics[width=\wid\linewidth]{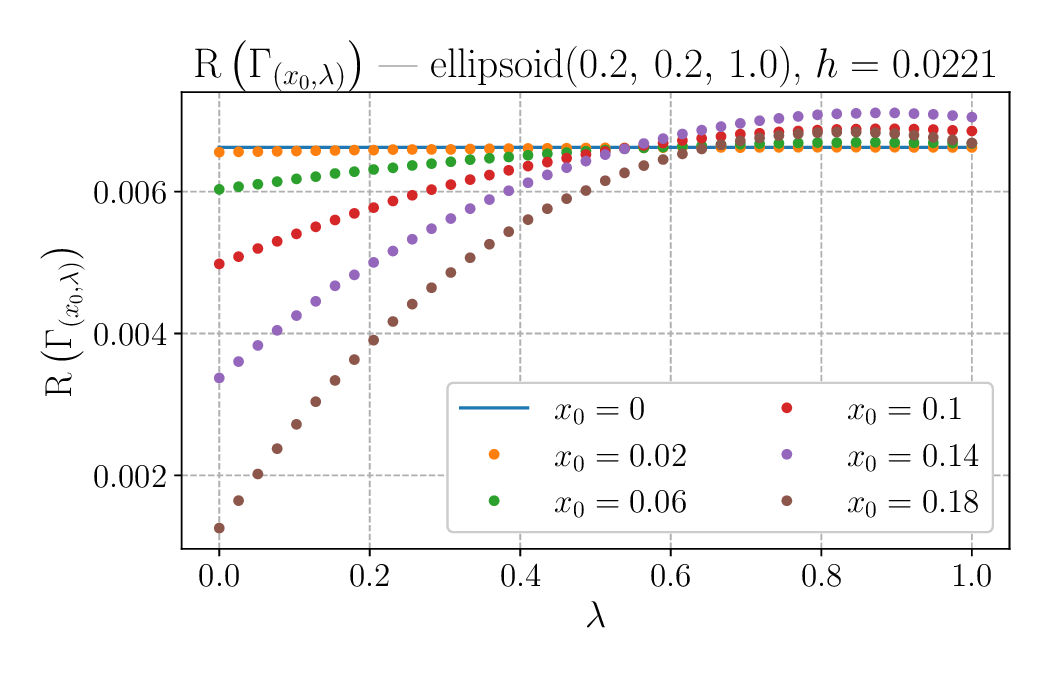}
        \caption{Ellipsoid with $a=0.2$.}
    \end{subfigure}
    \caption{Isoflux ratios of the competitors $\Gamma_{(x_0,\lambda)}$ for a sample of $x_0$-values. For the unit ball (left), none of the competitors considered outperforms the diameter aligned with the applied field. For the elongated ellipsoid (right), some off-axis competitors have a larger isoflux ratio than the major axis.}
    \label{fig:isoflux-ellipsoid}
\end{figure}

To identify the best competitor within the family considered for the ellipsoid, we evaluate $\R\bigl(\Gamma_{(x_0,\lambda)}\bigr)$ over a uniform $500\times500$ grid in the parameter space $(x_0,\lambda)$. The optimized parameters yield
$$
\R\bigl(\Gamma_{(x_0,\lambda)}\bigr)>\R\bigl(\Gamma_{(0,0)}\bigr),
$$
providing numerical evidence that the major axis is not optimal when the minor semi-axis is sufficiently small. The corresponding competitor is shown in Figure~\ref{fig:ellipsoid-full}.

\begin{figure}[ht!]
    \centering
    \includegraphics[width=\linewidth]{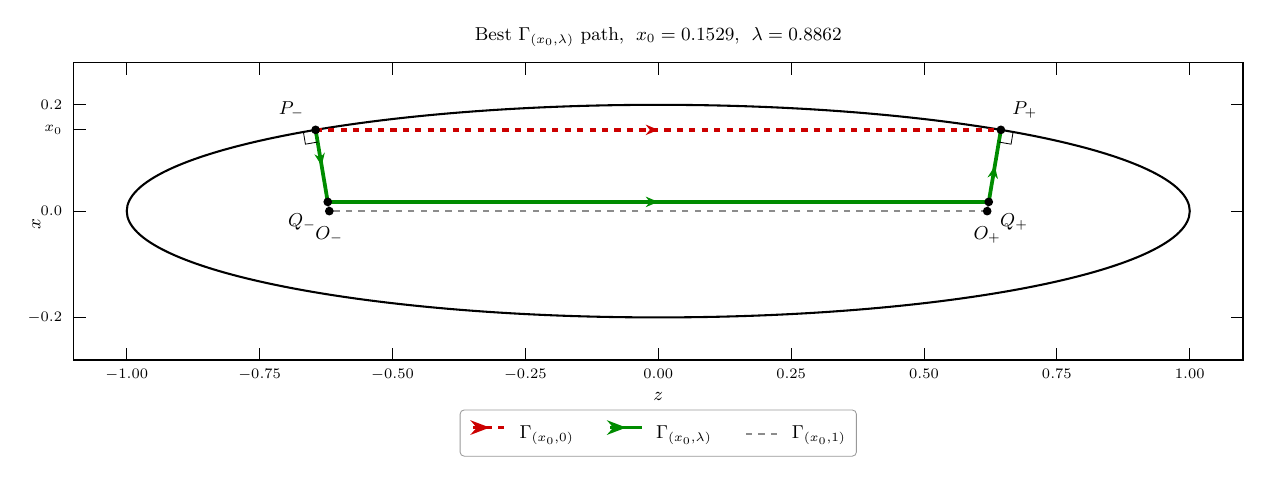}
    \caption{Best off-axis competitor within the family considered for the ellipsoid with $a=0.2$, with an isoflux ratio larger than that of the major axis.}
    \label{fig:ellipsoid-full}
\end{figure}

Although this comparison does not solve the full isoflux problem in ellipsoids, it provides numerical evidence that the major axis is not a maximizer when $a$ is sufficiently small. The geometry of the best competitors within the family suggests that an isoflux maximizer may instead take the form of a smooth U-shaped configuration, reminiscent of the U-vortices observed experimentally and numerically in rotating Bose--Einstein condensates \cites{BEC4,BEC3}; see also \cites{BEC1,BEC2,AftalionBook}. Moreover, rotating any such off-axis configuration around the major axis generates a continuous family of equivalent competitors. This suggests non-uniqueness and the presence of a degenerate rotational direction in the isoflux problem.

\appendix
\section{Explicit solution in a ball under a uniform normalized applied field}\label{sec:solball}
We recall the explicit solution of the London equation when $\Omega=B(0,R)$ is a ball of radius $R$ centered at the origin and $\Hzeroex=\hat z$ is the unit vector in the $z$-direction. This solution will serve as a benchmark for the numerical simulations. Using spherical coordinates $(r,\theta,\phi)$, where $r$ denotes the Euclidean distance from the origin, $\theta$ the azimuthal angle, and $\phi$ the polar angle, we obtain the following explicit formulas; see \cites{Lon,AlaBroMon}. 

For $r\geq R$,
$$
\Hzero=\hat z+ \frac{2M}{r^3}\cos\phi \hat r+\frac{M}{r^3}\sin\phi \hat \phi
=\left(1+\frac{2M}{r^3}\right)\cos\phi \hat r+\left(-1+\frac{M}{r^3}\right)\sin\phi \hat \phi,
$$  
where  
$$
M =-\frac{R^3}{2}\left(1-\frac{3}{R}\coth R+\frac{3}{R^2}\right).
$$  
For $r\leq R$, the expression becomes  
$$
\Hzero =\frac{3R}{r^2\sinh R}\left(\cosh r-\frac{\sinh r}r\right)\cos\phi \hat r
+\frac{3R}{2r^2\sinh R}\left(\cosh r-\frac{1+r^2}r \sinh r\right)\sin \phi \hat \phi.
$$  
The associated field $\Bzero$ is given by  
$$
\Bzero =-\frac{3R}{r^2\sinh R}\left(\cosh r-\frac{\sinh r}r\right)\cos\phi \hat r
-\frac{3R}{2r^2\sinh R}\left(\cosh r-\frac{1+r^2}r \sinh r\right)\sin \phi \hat \phi
-C\hat z,
$$  
where  
$$
C=\dfrac3{2R\sinh R}\left(\cosh R-\dfrac{1+R^2}R \sinh R\right).
$$
Finally, the vector field $\Uzero$ defined in~\eqref{eq:Uzero-definition} is given as follows. For $r\geq R$, 
$$
\Uzero=\frac{M}{r^2}\sin\phi\hat\theta,
$$
where $M$ is the same constant as above. For $r\leq R$,
$$
\Uzero=\frac{3R}{2\sinh R}\left(\frac{\cosh r}{r}-\frac{\sinh r}{r^2}\right)\sin\phi\hat\theta.
$$ 

\section{Boundary integral operators for the exterior problem}
\label{sec:BIOs-results}

This appendix collects the boundary integral operator framework used to handle
the exterior equation in \eqref{eq:Uzero-problem}. That is
\begin{equation}\label{eq:Uexterior}
	\begin{aligned}
		\curl\curl\Uzero
		&=0
		&&\text{in }\RR^3\setminus\overline{\Omega},\\
        \diver \Uzero &= 0
        &&\text{in }\RR^3\setminus\overline{\Omega},\\
		\lim\limits_{|x|\to\infty}|\Uzero(x)|
		&=0.&
	\end{aligned}
\end{equation}
We introduce the fundamental solution, layer potentials, and boundary integral operators. We then establish a representation formula for solutions of the exterior problem~\eqref{eq:Uexterior} that depends only on the boundary data of $\Uzero$. This allows us to derive the Calder\'on equations used in the \ac{fem-bem} coupling.

\subsection{Layer potentials and boundary integral operators}

The fundamental solution of the Laplace equation in $\RR^3$ is
$$G_0(x-y) = \frac{1}{4\pi|x-y|}.$$
Using $G_0$, we define the scalar single-layer potential, the vector
single-layer potential, and the vector double-layer potential by
\begin{equation*}
    \begin{aligned}
        \Psi_{\mathrm{SL}}(\lambda)(x) &\coloneqq
            \int_{\partial\Omega} G_0(x-y)\lambda(y)\,\dsy,
            && \lambda \in H^{-1/2}(\partial\Omega),\\
        \mathbf{\Psi}_{\mathrm{SL}}(\mathbf{u})(x) &\coloneqq
            \int_{\partial\Omega} G_0(x-y)\mathbf{u}(y)\,\dsy,
            && \mathbf{u} \in H^{-1/2}(\diver_{\partial\Omega}, \partial\Omega),\\
        \mathbf{\Psi}_{\mathrm{DL}}(\mathbf{u})(x) &\coloneqq
            \curl\int_{\partial\Omega} G_0(x-y)\mathbf{u}(y)\,\dsy,
            && \mathbf{u} \in H^{-1/2}(\diver_{\partial\Omega}, \partial\Omega).
    \end{aligned}
\end{equation*}
From these, we define the boundary integral operators
\begin{equation}\label{eq:BIOs}
    \begin{aligned}
        \mathbf{V}_0\lambda &\coloneqq
            \tfrac{1}{2}(\gammatau^+ + \gammatau^-)\mathbf{\Psi}_{\mathrm{SL}}\lambda,\\
        \mathbf{W}_0\lambda &\coloneqq
            \tfrac{1}{2}(\gammaN^+ + \gammaN^-)\mathbf{\Psi}_{\mathrm{DL}}\lambda,\\
        \mathbf{K}_0\lambda &\coloneqq
            \tfrac{1}{2}(\gammatau^+ + \gammatau^-)\mathbf{\Psi}_{\mathrm{DL}}\lambda.
    \end{aligned}
\end{equation}
These are bounded operators from $H^{-1/2}(\diver_{\partial\Omega}, \partial\Omega)$
to itself \cite{buffa2003boundary}. Moreover, $\mathbf{V}_0$ is elliptic and
$\mathbf{K}_0$ is compact on $H^{-1/2}(\diver_{\partial\Omega}, \partial\Omega)$; see
\cite{buffa2003boundary}.

\subsection{Representation formula}

We begin with two technical lemmas. The first is an identity relating the
normal trace of a curl to the surface divergence of the tangential trace.

\begin{lemma}\label{lemma:normal-curl}
    For $A \in H(\curl, \Omega)$, it holds
    \begin{equation}\label{div-surface}
        \curl A \cdot \nu|_{\partial\Omega} =
        \diver_{\partial\Omega}(A \times \nu|_{\partial\Omega}).
    \end{equation}
\end{lemma}

\begin{lemma}\label{div0}
    If $A \in H(\curl, \Omega) \cap H(\curl^2, \Omega)$ with $\curl^2 A = 0$,
    then $$\diver_{\partial\Omega}(\curl A \times \nu|_{\partial\Omega^-}) = 0.$$
\end{lemma}
    \begin{proof}
        We use the result of \eqref{div-surface} in Lemma~\ref{lemma:normal-curl} for the vector field $\curl A$. It follows that
        $$(\curl^2 A) \cdot \nu|_{\partial\Omega} = (\curl (\curl A)) \cdot \nu|_{\partial\Omega} =
        \diver_{\partial\Omega}(\curl A \times \nu|_{\partial\Omega^-}).$$
        Since $\curl^2 A = 0,$ we conclude that 
        $$ \diver_{\partial\Omega}(\curl A \times \nu|_{\partial\Omega^-}) = 0.$$
    \end{proof}
The result of Lemma~\ref{div0} illustrates that the natural space for $\curl \Uzero \times \nu |_{\partial\Omega^+}$, where $\Uzero$ satisfies \eqref{eq:Uexterior}, is the subspace of $H^{-1/2}(\diver_{\partial\Omega}, \partial\Omega)$ given by
\begin{equation}\label{eq:div0-trace-space}
 H^{-1/2}(\diver_{\partial\Omega}0, \partial\Omega) \coloneqq \left\{ \lambda\in H^{-1/2}(\diver_{\partial\Omega}, \partial\Omega) : \diver_{\partial\Omega}\lambda = 0  \right\}.   
\end{equation}
 
The following lemma relates the divergence of the vector single-layer potential
to the scalar single-layer potential applied to the surface divergence.

\begin{lemma}\label{lem:div-SL}
    For $\mathbf{u} \in H^{-1/2}(\diver_{\partial\Omega}, \partial\Omega)$,
    $$\diver\,\mathbf{\Psi}_{\mathrm{SL}}(\mathbf{u}) =
    \Psi_{\mathrm{SL}}(\diver_{\partial\Omega}\mathbf{u}).$$
\end{lemma}

We can now state the representation formula for the exterior problem. Similar results can also be found in \cites{buffa2003boundary,hiptmair2002symmetric}.
\begin{proposition}\label{prop:rep-formula}
      Any solution of \eqref{eq:Uexterior} admits the representation 
      \begin{equation*}
          \Uzero(x) = \mathbf{\Psi}_{\mathrm{DL}}(\gammatau^+\Uzero)(x)
                    - \mathbf{\Psi}_{\mathrm{SL}}(\gammaN^+\Uzero)(x)
                    + \nabla\Psi_{\mathrm{SL}}(\gamman^+\Uzero)(x),
          \qquad x \in \RR^3\setminus\overline{\Omega}.
      \end{equation*}
      Conversely, for any $\mathbf{u}, \mathbf{v} \in
      H^{-1/2}(\diver_{\partial\Omega}, \partial\Omega)$ and $\mu \in
      H^{-1/2}(\partial\Omega)$, the field
      $$V(x) \coloneqq \mathbf{\Psi}_{\mathrm{DL}}(\mathbf{u})(x)
                - \mathbf{\Psi}_{\mathrm{SL}}(\mathbf{v})(x)
                + \nabla\Psi_{\mathrm{SL}}(\mu)(x)$$
      satisfies $\curl\curl V = 0$ in $\RR^3\setminus\overline{\Omega}$ and
      $|V(x)|\to 0$ as $|x|\to\infty$.
  \end{proposition}

  \begin{corollary}
      The field $V$ defined in Proposition~\ref{prop:rep-formula} is additionally
      divergence-free in $\RR^3\setminus\overline{\Omega}$ if and only if
      $\diver_{\partial\Omega}\mathbf{v} = 0$.
      \begin{proof}
          Since $\mathbf{\Psi}_{\mathrm{DL}}(\mathbf{u})$ is the image of a $\curl$ operator, we have
          $\diver\mathbf{\Psi}_{\mathrm{DL}}(\mathbf{u}) = 0$. Since $G_0$ is
          harmonic outside $\partial\Omega$, $\diver\nabla\Psi_{\mathrm{SL}}(\mu)
          = \Delta\Psi_{\mathrm{SL}}(\mu) = 0$. For the remaining term,
          Lemma~\ref{lem:div-SL} gives $\diver\mathbf{\Psi}_{\mathrm{SL}}(\mathbf{v})
          = \Psi_{\mathrm{SL}}(\diver_{\partial\Omega}\mathbf{v})$, which vanishes
          if and only if $\diver_{\partial\Omega}\mathbf{v} = 0$. In particular,
          taking $\mathbf{v} = \gammaN^+\Uzero$, Lemma~\ref{div0} guarantees
          $\diver_{\partial\Omega}(\gammaN^+\Uzero) = 0$ whenever
          $\curl\curl\Uzero = 0$.
      \end{proof}
  \end{corollary}

\subsection{Calder\'on equations}

Applying the trace operators $\gammatau^+$ and $\gammaN^+$ to
\eqref{eq:rep-formula} and using the jump relations for the layer potentials
\cite{buffa2003boundary}, one obtains two Calder\'on equations. Applying
$\gammatau^+$ yields
\begin{equation*}
    \gammatau^+\Uzero =
    \left(\tfrac{1}{2}\mathbf{I} + \mathbf{K}_0\right)(\gammatau^+\Uzero)
    - \mathbf{V}_0(\gammaN^+\Uzero)
    + \bcurl_{\partial\Omega}\Psi_{\mathrm{SL}}(\gamman^+\Uzero),
\end{equation*}
which rearranges to
\begin{equation*}
    \left(\tfrac{1}{2}\mathbf{I} - \mathbf{K}_0\right)(\gammatau^+\Uzero)
    + \mathbf{V}_0(\gammaN^+\Uzero)
    - \bcurl_{\partial\Omega}\Psi_{\mathrm{SL}}(\gamman^+\Uzero) = 0.
\end{equation*}
Applying $\gammaN^+$ yields
\begin{equation*}
    \gammaN^+\Uzero =
    \mathbf{W}_0(\gammatau^+\Uzero)
    + \left(\tfrac{1}{2}\mathbf{I} - \mathbf{K}_0\right)(\gammaN^+\Uzero),
\end{equation*}
which rearranges to
\begin{equation}\label{eq:calderon2Bis}
    -\mathbf{W}_0(\gammatau^+\Uzero)
    + \left(\tfrac{1}{2}\mathbf{I} + \mathbf{K}_0\right)(\gammaN^+\Uzero) = 0.
\end{equation}
Equation \eqref{eq:calderon2Bis} is a second-kind boundary integral equation for $\gammaN^+\Uzero$,
since $\mathbf{K}_0$ is compact \cite{buffa2003boundary}. This is the one
we use in the \ac{fem-bem} coupling; see Section~\ref{sec:numerical}.

\section*{Acknowledgments}
N.B was supported by Centro de Modelamiento Matem\'atico (CMM), Proyecto Basal FB210005. C.R. was supported by ANID FONDECYT 1231593.

\bibliography{references}

\begin{bibdiv}
\begin{biblist}

\bib{alnaes2015fenics}{article}{
      author={Aln{\ae}s, Martin~S.},
      author={Blechta, Jan},
      author={Hake, Johan},
      author={Johansson, August},
      author={Kehlet, Benjamin},
      author={Logg, Anders},
      author={Richardson, Chris~N.},
      author={Ring, Johannes},
      author={Rognes, Marie~E.},
      author={Wells, Garth~N.},
       title={The {FEniCS} project version 1.5},
        date={2015},
     journal={Arch. Numer. Softw.},
      volume={3},
      number={100},
       pages={9\ndash 23},
}

\bib{AlaBroMon}{article}{
      author={Alama, Stan},
      author={Bronsard, Lia},
      author={Montero, J.~Alberto},
       title={On the {G}inzburg-{L}andau model of a superconducting ball in a
  uniform field},
        date={2006},
        ISSN={0294-1449},
     journal={Ann. Inst. H. Poincar\'e Anal. Non Lin\'eaire},
      volume={23},
      number={2},
       pages={237\ndash 267},
         url={http://dx.doi.org/10.1016/j.anihpc.2005.03.004},
      review={\MR{2201153}},
}

\bib{BEC3}{article}{
      author={Aftalion, Amandine},
      author={Danaila, Ionut},
       title={Three-dimensional vortex configurations in a rotating
  {B}ose-{E}instein condensate},
        date={2003},
     journal={Phys. Rev. A},
      volume={68},
      number={2},
       pages={023603},
         url={https://doi.org/10.1103/PhysRevA.68.023603},
}

\bib{AftalionBook}{book}{
      author={Aftalion, Amandine},
       title={Vortices in {B}ose-{E}instein condensates},
      series={Progress in Nonlinear Differential Equations and their
  Applications},
   publisher={Birkh{\"a}user Boston, Inc.},
     address={Boston, MA},
        date={2006},
      volume={67},
        ISBN={978-0-8176-4392-8},
         url={https://doi.org/10.1007/0-8176-4492-X},
}

\bib{BEC2}{article}{
      author={Aftalion, Amandine},
      author={Jerrard, Robert~L.},
       title={Shape of vortices for a rotating {B}ose-{E}instein condensate},
        date={2002},
     journal={Phys. Rev. A},
      volume={66},
      number={2},
       pages={023611},
         url={https://doi.org/10.1103/PhysRevA.66.023611},
}

\bib{BEC1}{article}{
      author={Aftalion, Amandine},
      author={Rivi{\`e}re, Tristan},
       title={Vortex energy and vortex bending for a rotating {B}ose-{E}instein
  condensate},
        date={2001},
     journal={Phys. Rev. A},
      volume={64},
      number={4},
       pages={043611},
         url={https://doi.org/10.1103/PhysRevA.64.043611},
}

\bib{arnold2018finite}{book}{
      author={Arnold, Douglas~N.},
       title={Finite element exterior calculus},
      series={CBMS-NSF Regional Conference Series in Applied Mathematics},
   publisher={Society for Industrial and Applied Mathematics (SIAM),
  Philadelphia, PA},
        date={2018},
      volume={93},
        ISBN={978-1-611975-53-6},
         url={https://doi.org/10.1137/1.9781611975543.ch1},
      review={\MR{3908678}},
}

\bib{boffi2013mixed}{book}{
      author={Boffi, Daniele},
      author={Brezzi, Franco},
      author={Fortin, Michel},
       title={Mixed finite element methods and applications},
      series={Springer Series in Computational Mathematics},
   publisher={Springer, Heidelberg},
        date={2013},
      volume={44},
        ISBN={978-3-642-36518-8; 978-3-642-36519-5},
         url={https://doi.org/10.1007/978-3-642-36519-5},
      review={\MR{3097958}},
}

\bib{buffa2007dual}{article}{
      author={Buffa, Annalisa},
      author={Christiansen, Snorre~H.},
       title={A dual finite element complex on the barycentric refinement},
        date={2007},
        ISSN={0025-5718,1088-6842},
     journal={Math. Comp.},
      volume={76},
      number={260},
       pages={1743\ndash 1769},
         url={https://doi.org/10.1090/S0025-5718-07-01965-5},
      review={\MR{2336266}},
}

\bib{buffa2002traces}{article}{
      author={Buffa, A.},
      author={Costabel, M.},
      author={Sheen, D.},
       title={On traces for {${\bf H}({\bf curl},\Omega)$} in {L}ipschitz
  domains},
        date={2002},
        ISSN={0022-247X,1096-0813},
     journal={J. Math. Anal. Appl.},
      volume={276},
      number={2},
       pages={845\ndash 867},
         url={https://doi.org/10.1016/S0022-247X(02)00455-9},
      review={\MR{1944792}},
}

\bib{BCS}{article}{
      author={Bardeen, J.},
      author={Cooper, L.~N.},
      author={Schrieffer, J.~R.},
       title={Theory of superconductivity},
        date={1957},
     journal={Phys. Rev.},
      volume={108},
       pages={1175\ndash 1204},
         url={https://link.aps.org/doi/10.1103/PhysRev.108.1175},
}

\bib{buffa2003boundary}{article}{
      author={Buffa, A.},
      author={Hiptmair, R.},
      author={von Petersdorff, T.},
      author={Schwab, C.},
       title={Boundary element methods for {M}axwell transmission problems in
  {L}ipschitz domains},
        date={2003},
        ISSN={0029-599X,0945-3245},
     journal={Numer. Math.},
      volume={95},
      number={3},
       pages={459\ndash 485},
         url={https://doi.org/10.1007/s00211-002-0407-z},
      review={\MR{2012928}},
}

\bib{BalJerOrlSon2}{article}{
      author={Baldo, S.},
      author={Jerrard, R.~L.},
      author={Orlandi, G.},
      author={Soner, H.~M.},
       title={Vortex density models for superconductivity and superfluidity},
        date={2013},
        ISSN={0010-3616},
     journal={Comm. Math. Phys.},
      volume={318},
      number={1},
       pages={131\ndash 171},
         url={http://dx.doi.org/10.1007/s00220-012-1629-2},
      review={\MR{3017066}},
}

\bib{Brandt1994}{article}{
      author={Brandt, Ernst~Helmut},
       title={Thin superconductors in a perpendicular magnetic ac field:
  General formulation and strip geometry},
        date={1994},
     journal={Physical Review B},
      volume={49},
      number={13},
       pages={9024\ndash 9040},
}

\bib{Brandt1996}{article}{
      author={Brandt, Ernst~Helmut},
       title={Superconductors of finite thickness in a perpendicular magnetic
  field: Strips and slabs},
        date={1996},
     journal={Physical Review B},
      volume={54},
      number={6},
       pages={4246\ndash 4264},
}

\bib{brenner2008mathematical}{book}{
      author={Brenner, Susanne~C.},
      author={Scott, L.~Ridgway},
       title={The mathematical theory of finite element methods},
     edition={Third},
      series={Texts in Applied Mathematics},
   publisher={Springer, New York},
        date={2008},
      volume={15},
        ISBN={978-0-387-75933-3},
         url={https://doi.org/10.1007/978-0-387-75934-0},
      review={\MR{2373954}},
}

\bib{Betcke2021}{article}{
      author={Betcke, Timo},
      author={Scroggs, Matthew~W.},
       title={Bempp-cl: A fast python based just-in-time compiling boundary
  element library.},
        date={2021},
     journal={Journal of Open Source Software},
      volume={6},
      number={59},
       pages={2879},
         url={https://doi.org/10.21105/joss.02879},
}

\bib{balay2019petsc}{techreport}{
      author={Balay, Satish},
      author={others},
       title={{PETSc} users manual},
 institution={Argonne National Laboratory},
        date={2019},
      number={ANL-95/11, Revision 3.11},
         url={https://www.osti.gov/biblio/1577437},
}

\bib{ClemBrandt2005}{article}{
      author={Clem, John~R.},
      author={Brandt, Ernst~Helmut},
       title={Response of thin-film {SQUIDs} to applied fields and vortex
  fields: Linear {SQUIDs}},
        date={2005},
     journal={Physical Review B},
      volume={72},
      number={17},
       pages={174511},
}

\bib{elasmi2021johnson}{unpublished}{
      author={Elasmi, Mehdi},
      author={Erath, Christoph},
      author={Kurz, Stefan},
       title={The {J}ohnson--{N}\'ed\'elec {FEM}--{BEM} coupling for
  magnetostatic problems in the isogeometric framework},
        date={2021},
        note={Available at
  \href{https://arxiv.org/abs/2110.04150}{arXiv:2110.04150}},
}

\bib{FHSS}{article}{
      author={Frank, Rupert~L.},
      author={Hainzl, Christian},
      author={Seiringer, Robert},
      author={Solovej, Jan~Philip},
       title={Microscopic derivation of {G}inzburg-{L}andau theory},
        date={2012},
        ISSN={0894-0347},
     journal={J. Amer. Math. Soc.},
      volume={25},
      number={3},
       pages={667\ndash 713},
         url={https://doi.org/10.1090/S0894-0347-2012-00735-8},
      review={\MR{2904570}},
}

\bib{GinLan}{article}{
      author={Ginzburg, V.~L.},
      author={Landau, L.~D.},
       title={On the theory of superconductivity},
        date={1950},
     journal={Zh. Eksp. Teor. Fiz.},
      volume={20},
       pages={1064\ndash 1082},
        note={English translation in: \emph{Collected papers of L.D.Landau},
  Edited by D. Ter. Haar, Pergamon Press, Oxford 1965, pp. 546--568},
}

\bib{hackbusch2013multi}{book}{
      author={Hackbusch, Wolfgang},
       title={Multi-grid methods and applications},
      series={Springer Series in Computational Mathematics},
   publisher={Springer-Verlag},
     address={Berlin},
        date={1985},
      volume={4},
}

\bib{hiptmair2002finite}{article}{
      author={Hiptmair, R.},
       title={Finite elements in computational electromagnetism},
        date={2002},
        ISSN={0962-4929,1474-0508},
     journal={Acta Numer.},
      volume={11},
       pages={237\ndash 339},
         url={https://doi.org/10.1017/S0962492902000041},
      review={\MR{2009375}},
}

\bib{hiptmair2002symmetric}{article}{
      author={Hiptmair, R.},
       title={Symmetric coupling for eddy current problems},
        date={2002},
        ISSN={0036-1429,1095-7170},
     journal={SIAM J. Numer. Anal.},
      volume={40},
      number={1},
       pages={41\ndash 65},
         url={https://doi.org/10.1137/S0036142900380467},
      review={\MR{1921909}},
}

\bib{fumio1987mixed}{inproceedings}{
      author={Kikuchi, Fumio},
       title={Mixed and penalty formulations for finite element analysis of an
  eigenvalue problem in electromagnetism},
        date={1987},
   booktitle={Proceedings of the first world congress on computational
  mechanics ({A}ustin, {T}ex., 1986)},
      volume={64},
       pages={509\ndash 521},
         url={https://doi.org/10.1016/0045-7825(87)90053-3},
      review={\MR{912525}},
}

\bib{KhapaevEtAl2003}{article}{
      author={Khapaev, M.~M.},
      author={Kupriyanov, M.~Yu.},
      author={Goldobin, E.},
      author={Siegel, M.},
       title={Current distribution simulation for superconducting multi-layer
  structures},
        date={2003},
     journal={Superconductor Science and Technology},
      volume={16},
      number={1},
       pages={24\ndash 27},
}

\bib{London1935}{article}{
      author={London, F.},
      author={London, H.},
       title={The electromagnetic equations of the superconductor},
        date={1935},
     journal={Proceedings of the Royal Society of London. Series A},
      volume={149},
       pages={71\ndash 88},
}

\bib{Lon}{book}{
      author={London, F.},
       title={Superfluids},
      series={Structure of matter series},
   publisher={Wiley, New York},
        date={1950},
      number={v. 1},
         url={https://books.google.de/books?id=DtvvAAAAMAAJ},
}

\bib{monk2003finite}{book}{
      author={Monk, Peter},
       title={Finite element methods for {M}axwell's equations},
      series={Numerical Mathematics and Scientific Computation},
   publisher={Oxford University Press, New York},
        date={2003},
        ISBN={0-19-850888-3},
         url={https://doi.org/10.1093/acprof:oso/9780198508885.001.0001},
      review={\MR{2059447}},
}

\bib{BEC4}{article}{
      author={Rosenbusch, P.},
      author={Bretin, V.},
      author={Dalibard, J.},
       title={Dynamics of a single vortex line in a {B}ose-{E}instein
  condensate},
        date={2002},
     journal={Phys. Rev. Lett.},
      volume={89},
      number={20},
       pages={200403},
         url={https://doi.org/10.1103/PhysRevLett.89.200403},
}

\bib{Rom2}{article}{
      author={Rom\'{a}n, Carlos},
       title={On the first critical field in the three dimensional
  {G}inzburg-{L}andau model of superconductivity},
        date={2019},
        ISSN={0010-3616},
     journal={Comm. Math. Phys.},
      volume={367},
      number={1},
       pages={317\ndash 349},
         url={https://doi.org/10.1007/s00220-019-03306-w},
      review={\MR{3933412}},
}

\bib{RSS}{article}{
      author={Rom\'an, Carlos},
      author={Sandier, Etienne},
      author={Serfaty, Sylvia},
       title={Bounded vorticity for the 3{D} {G}inzburg-{L}andau model and an
  isoflux problem},
        date={2023},
        ISSN={0024-6115,1460-244X},
     journal={Proc. Lond. Math. Soc. (3)},
      volume={126},
      number={3},
       pages={1015\ndash 1062},
         url={https://doi.org/10.1112/plms.12505},
      review={\MR{4563866}},
}

\bib{RSS2}{unpublished}{
      author={Rom\'{a}n, Carlos},
      author={Sandier, Etienne},
      author={Serfaty, Sylvia},
       title={Vortex lines interaction in the three-dimensional magnetic
  {G}inzburg--{Landau} model},
        date={2025},
        note={Available at
  \href{https://arxiv.org/abs/2510.14910}{arXiv:2510.14910}},
}

\bib{sauter2010boundary}{book}{
      author={Sauter, Stefan~A.},
      author={Schwab, Christoph},
       title={Boundary element methods},
      series={Springer Series in Computational Mathematics},
   publisher={Springer-Verlag, Berlin},
        date={2011},
      volume={39},
        ISBN={978-3-540-68092-5},
         url={https://doi.org/10.1007/978-3-540-68093-2},
        note={Translated and expanded from the 2004 German original},
      review={\MR{2743235}},
}

\bib{saad1986gmres}{article}{
      author={Saad, Youcef},
      author={Schultz, Martin~H.},
       title={G{MRES}: a generalized minimal residual algorithm for solving
  nonsymmetric linear systems},
        date={1986},
        ISSN={0196-5204},
     journal={SIAM J. Sci. Statist. Comput.},
      volume={7},
      number={3},
       pages={856\ndash 869},
         url={https://doi.org/10.1137/0907058},
      review={\MR{848568}},
}

\bib{steinbach2011note}{article}{
      author={Steinbach, O.},
       title={A note on the stable one-equation coupling of finite and boundary
  elements},
        date={2011},
        ISSN={0036-1429,1095-7170},
     journal={SIAM J. Numer. Anal.},
      volume={49},
      number={4},
       pages={1521\ndash 1531},
         url={https://doi.org/10.1137/090762701},
      review={\MR{2831059}},
}

\end{biblist}
\end{bibdiv}

\end{document}